\documentclass[12pt,a4paper]{article}
\usepackage{amsmath,amssymb,amsthm,bbm,enumitem,fullpage,parskip,tikz,mathtools,thm-restate}
\usepackage[hidelinks]{hyperref}
\hypersetup{colorlinks=true,linktoc=all,allcolors=black}
\usepackage[capitalise]{cleveref}
\usepackage[T1]{fontenc}

\newtheorem{theorem}{Theorem}
\newtheorem{lemma}[theorem]{Lemma}
\newtheorem{corollary}[theorem]{Corollary}
\theoremstyle{remark}
\newtheorem{remark}[theorem]{Remark}

\newcommand\ol[1]{\mathchoice
{\vphantom{\bar{#1}}\rlap{$\mkern1.5mu\overline{\phantom{\raise-0.5pt\hbox{$#1\mkern-2mu$}}}$}#1}
{\vphantom{\bar{#1}}\rlap{$\mkern1.5mu\overline{\phantom{\raise-0.5pt\hbox{$#1\mkern-2mu$}}}$}#1}
{\vphantom{\bar{#1}}\rlap{$\scriptstyle\mkern1.5mu\overline{\phantom{\raise-0.5pt\hbox{$\scriptstyle #1\mkern-2mu$}}}$}#1}
{\vphantom{\bar{#1}}\rlap{$\scriptscriptstyle\mkern1.5mu\overline{\phantom{\raise-0.5pt\hbox{$\scriptscriptstyle#1\mkern-2mu$}}}$}#1}{}}

\newcommand\N{\mathbb{N}}
\newcommand\Z{\mathbb{Z}}
\newcommand\R{\mathbb{R}}
\newcommand\one{\mathbbm{1}}
\newcommand\eps{\varepsilon}
\newcommand\E{\mathbb{E}}
\newcommand\Prb{\mathbb{P}}
\newcommand\Var{\mathrm{Var}}
\newcommand\Cov{\mathrm{Cov}}
\newcommand\Bin{\mathrm{Bin}}
\newcommand\Tr{\mathrm{Tr}}
\newcommand\Nor{N}
\newcommand\bk{\ol{k}}
\newcommand\bK{\ol{K}}
\newcommand\bS{\ol{S}}
\newcommand\bW{\ol{W}}
\newcommand\bw{\ol{w}}
\newcommand\bv{\ol{v}}
\newcommand\bx{\ol{x}}
\newcommand\by{\ol{y}}
\newcommand\bz{\ol{z}}

\newcommand\vW{\boldsymbol{W}}
\newcommand\vw{\boldsymbol{w}}
\newcommand\vZ{\boldsymbol{Z}}
\newcommand\cF{\mathcal{F}}

\newcommand\cI{\mathcal{I}}
\newcommand\cT{\mathcal{T}}
\newcommand\wvhp{w.v.h.p.}

\title{A local limit theorem for the edge counts of random induced subgraphs of a random graph}
\author{Paul Balister\footnote{Mathematical Institute, University of Oxford, Oxford OX2\thinspace6GG, United Kingdom
\texttt{\{balister,powierski,scott,jane.tan\}@maths.ox.ac.uk}}~\footnote{Research supported by EPSRC grant EP/W015404/1.} \;
Emil Powierski\protect\footnotemark[1] \;
Alex Scott\protect\footnotemark[1]~\footnote{Research supported by EPSRC grant EP/X013642/1.\newline
For the purpose of Open Access, the authors have applied a CC BY public copyright licence to any Author Accepted Manuscript (AAM) version arising from this submission.} \;
Jane Tan\protect\footnotemark[1]}
\date{}

\begin{document}

\maketitle

\begin{abstract}
Consider a `dense' Erd\H{o}s--R\'enyi random graph model $G=G_{n,M}$ with $n$ vertices and $M$ edges, where we assume the
edge density $M/\binom{n}{2}$ is bounded away from 0 and~1. Fix $k=k(n)$ with $k/n$ bounded away from 0 and~1, and let $S$
be a random subset of size $k$ of the vertices of~$G$. We show that with probability $1-\exp(-n^{\Omega(1)})$,
$G$ satisfies both a central limit theorem and a local limit theorem for the empirical distribution of the edge count
$e(G[S])$ of the subgraph of $G$ induced by~$S$, where the distribution is over uniform random choices of the $k$-set~$S$.
\end{abstract}

\section{Introduction}\label{sec:intro}
Consider an Erd\H{o}s-R\'enyi random graph model $G=G_{n,M}$ with $n$ vertices and $M$ edges.
Set $N:=\binom{n}{2}$ and assume $M/N\in[\delta,1-\delta]$ for some fixed $\delta>0$, so that $M/N$
is bounded away from both 0 and 1 and both $G$ and its complement are dense.
Fix $k=k(n)$ and assume $k/n\in[\delta,1-\delta]$ is also bounded away from 0 and~1.
Letting $S$ be a uniformly chosen random subset of the vertices
of $G$ of size~$k$, we will study the edge count $e(S):=e(G[S])$ of the subgraph of $G$ induced by~$S$.
Throughout, we shall always assume that the number of vertices $n$ is sufficiently large, and $M$,
$k$ and $S$ are defined as above unless otherwise indicated within the scope of specific statements.
We say that an event $E$ happens in $G_{n,M}$ \emph{with very high probability} (w.v.h.p.)
if there exists $\eps>0$ such that, for sufficiently large~$n$, we have $\Prb(E) \ge 1-\exp(-n^\eps)$.
In this paper, we show that with very high probability (in choice of~$G$), the empirical
statistics of $e(S)$ satisfy both a central limit theorem and a local limit theorem. 

We will be considering probabilities in various probability spaces. Notably, this will involve probabilities
with respect to the random graph models $G_{n,p}$ or $G_{n,M}$, as well as probabilities with respect to the
random vertex set $S$ that we choose. Where helpful, we will use subscripts such as $\Prb_G$ for probabilities
over the choice of the random graph, and $\Prb_S$ for probabilities over the choice of $S$ conditioned on a
fixed choice of random graph~$G$.

By using a multidimensional version of Stein's method as described in~\cite{RR}, we first deduce in \cref{sec:clt} a central
limit theorem for $e(S)$. Define 
\begin{equation}\label{eq:lambda}
 \lambda:=\frac{(n^2-k^2)k^2}{2n^4}\cdot \frac{M(N-M)}{N^2}.
\end{equation}
and note that $\lambda=\Theta(1)$. Set $K:=\binom{k}{2}$. 

\begin{restatable}{theorem}{clt}\label{thm:clt}
 Assume that for some fixed $\delta>0$, $M/N,k/n\in[\delta,1-\delta]$.
 Then for any $\eps>0$ and w.v.h.p.\ in $G=G_{n,M}$ we have for all\/ $z\in\R$,
 \begin{equation}\label{eq:clt}
  \big|\Prb_S(e(S)\le z)-\Prb(Z\le z)\big|\le n^{-1/4+\eps},
 \end{equation}
 where $S \subseteq V(G)$ is a uniformly chosen random subset of size $k$ and\/ 
 $Z\sim \Nor(KM/N,\lambda n^2)$ is a normal random variable with mean $KM/N$ and
 variance $\lambda n^2$ with\/ $\lambda$ given by~\eqref{eq:lambda}.
\end{restatable}

This central limit theorem serves as a starting point for an iterative smoothing argument that we employ in \cref{sec:llt}
to descend to the following local limit theorem, which is our main result.

\begin{restatable}{theorem}{llt}\label{thm:llt}
 Assume that for some fixed $\delta>0$, $M/N,k/n\in[\delta,1-\delta]$.
 Then for any $\eps>0$ and w.v.h.p.\ in $G=G_{n,M}$ we have that for all integers $z$
 \[
  n\big|\Prb_S(e(S)=z)-\varphi(z)\big| \le n^{-1/14+\eps},
 \]
 where $\varphi$ is the density function of a normal random variable $Z\sim\Nor(KM/N,\lambda n^2)$.
\end{restatable}

In particular, this shows that the point probabilities of $e(S)$ are of order $1/n$ around its peak,
i.e., within $O(n)$ of its mean $KM/N$, and the relative error is at most $n^{-1/14+\eps}$ in this range.

Somewhat surprisingly, \cref{thm:clt} and \cref{thm:llt} are the first such results on the empirical
distribution of edge counts in this very natural setting of subgraphs induced by $k$-subsets,
although there is a good deal of related work which we discuss in \cref{sec:relatedwork}.

Results that hold w.v.h.p.\ can in general be easily transferred between the above setting
in $G_{n,M}$ and $G_{n,p}$ for $p\in[\delta,1-\delta]$; given a statement holding in $G_{n,M}$ and 
setting $p=M/N$, there is a polynomial probability (on the order of~$1/n$) that $e(G_{n,p})=M$.
In that case we have $(G_{n,p}\mid e(G_{n,p})=M)\sim G_{n,M}$, so the conclusion of
the smoothing lemma (\cref{lem:smoothing} below) transfers to $G_{n,M}$ again with
very small failure probability. Consequently, it will be convenient to prove most of our lemmas
en route w.v.h.p.\ for $G_{n,p}$.

However, we remark that the edge count of $G_{n,p}$ can typically vary by order $n$ and so $KM/N$ also varies
by order~$n$, which is of the same order as the standard deviation of~$Z$. Hence, we can't just replace $M$
by $p \binom{n}{2}$ and obtain the same results as above when working in~$G_{n,p}$.

We also remark that, due to the exponentially small failure probability, a union bound implies slightly
stronger versions of the above results which say that for fixed~$\delta>0$, w.v.h.p.~$G_{n,M}$ satisfies
the desired properties for all $k,M$ with $k/n,M/N \in [\delta,1-\delta]$. For the most part, it will be
more convenient for us to view $k$ and $M$ as fixed, but analogous union bound arguments will later come
up for other parameters.

\subsection{Discussion}
\label{sec:relatedwork}

Both Theorems \ref{thm:clt} and \ref{thm:llt} are concerned with induced subgraphs of \emph{fixed\/} size.
We can also consider a \emph{random size} model where we take a subgraph induced by a random subset~$S$,
where each vertex independently belongs to $S$ with constant probability~$r$.
In this case, it is straightforward to show that \wvhp\ $e(S)$ satisfies a central limit theorem.
We write $e(S)$ in the form $\sum_{i<j}a_{ij}X_iX_j$,
where $A=(a_{ij})$ is the adjacency matrix and the $X_i$ are independent Bernoulli random variables with mean~$r$.
Then a standard sufficient condition for asymptotic normality is that 
\begin{equation}\label{classical}
 \frac{\Tr(A^4)}{\sigma^4}\to0\qquad\text{ and }\qquad\frac{\max_i\sum_ja_{ij}^2}{\sigma^2}\to0,
\end{equation}
where $\sigma^2$ is the variance of $e(S)$ (see \cite{chatterjee,dejong,hall,rotar1974some}
for more details, including error bounds).
This is immediate in the random size model, as $\sigma^2$ is of order $n^3$ \wvhp\ (driven by copies of $K_{1,2}$),
while $\Tr(A^4)=O(n^4)$ and $\max_i\sum_ja_{ij}^2=\Delta(G)\le n$.
In fact, this holds for any sequence of graphs with density bounded away from 0 (as these contain
quadratically many copies of $K_{1,2}$).\footnote{In our case, where the size of $S$ is fixed,
\eqref{classical} does not apply, as the $X_i$ are not independent
(and, in any case, the first expression in \eqref{classical} does not tend to~0, as the variance is much smaller,
having order~$n^2$).} 
However, we cannot hope for a local limit theorem when
the size of $S$ varies. Indeed, a change of size of $S$ of just 1 will move the mean in Theorem~\ref{thm:llt}
by order~$n$, which is the same order as the standard deviation for fixed~$|S|$. Thus the distribution
will approach a rapidly varying function which is the sum of many small Gaussians (see~\cite{kwan}). 

There has been previous work on monochromatic edges in random $c$-colourings of random graphs: 
for a fixed~$c$, consider a function $\phi\colon V(G)\to[c]$ chosen uniformly at random and
let $Y=Y_G(\phi)$ be the number of monochromatic edges. For $G\in G_{n,p}$, the variance of $Y$
is typically of order $n^2$ (as copies of $K_{1,2}$ no longer contribute anything),
and so \eqref{classical} is no longer satisfied. In fact, the distribution of $Y$ is typically not Gaussian
(a Gaussian distribution requires the graph to have very few copies of~$C_4$: see \cite{bhattacharya} for a detailed analysis).  
Note that the colour classes in a random colouring do not have fixed size, so we are in a random size model.
However, while we do not pursue this further here, our methods should be applicable with fixed sizes of colour classes.
For example, with $c=2$ and colour classes of equal size, we would be considering a random bisection. 

A broader class of examples is given by Ramsey graphs. A graph with $n$ vertices is \emph{$C$-Ramsey}
if it has no clique or independent set of size $C\log_2n$.  It follows from a result of Erd\H os and
Szemer\'edi~$\cite{esz}$ that, for large~$n$, $C$-Ramsey graphs have densities bounded away from 0 and~1;
thus by~\eqref{classical}, in the `random size' model, $e(S)$ satisfies a central limit theorem.
However, the local picture is more subtle.  The first question here is whether there are \emph{any}
induced subgraphs with the sizes we require. For random graphs, this was shown by Calkin, Frieze and
McKay~\cite{calkin1992subgraph}, who proved that for $p$ fixed, with high probability a graph in $G_{n,p}$
contains induced subgraphs of all sizes between $0$ and $e(G)-cn^{3/2}/\sqrt{\log n}$. For Ramsey graphs
this is much harder: Erd\H os and McKay (see~\cite{efav}) conjectured that there are induced subgraphs
of all sizes up to~$cn^2$. After a long line of subsequent work, a much stronger statement was proved
by Kwan, Sah, Sauermann and Sawhney~\cite{kwan}: for every $\eps,C>0$, every $C$-Ramsey graph with
sufficiently many vertices has induced subgraphs of all sizes up to $(1-\eps)e(G)$.  But what about the
distribution of random subgraphs? Kwan, Sah, Sauermann and Sawhney showed further that, for any $C>0$,
and $r=r(n)$ bounded away from $0$ and~$1$, in the random size model we have 
\begin{equation}\label{ubr}
 \Prb[e(S)=x]\le A/\sigma,
\end{equation}
while for all $x$ within a bounded number of standard deviations of $\E e(S)$,
\[
 \Prb[e(S)=x]\ge a/\sigma,
\]
where $a$ and $A$ are positive constants. However, as noted in \cite{kwan}, a local limit need not hold in this setting.

Finally, we note that it would be interesting to prove analogous results for sparse graphs,
and for the distribution of other subgraphs (for example, the number of triangles in a random induced subgraph).

\section{The central limit theorem}\label{sec:clt}

The proof of our main central limit theorem is given in two parts. The bulk of the work occurs
in the first, where we derive an initial central limit theorem using Stein's method.
In the second part, we reformulate this statement to obtain \cref{thm:clt}.

It will be helpful to fix some notation that we will use throughout the section.
We will be working with a random graph $G=G_{n,M}$ with vertex set $V=V(G)$.
For a set $S\subseteq V$ of size $k$, we write $\ol{S}:=V\setminus S$ and $\ol{k}:=n-k$.
We also write $e(S)$ for the number of edges in the induced subgraph $G[S]$.
For a vertex $x\in V$, we will write $d_S(x)$ for the number of edges between $x$ and~$S$.
We also write $d(G)$ for the average degree of $G$. 

\subsection{Stein's method}

We will be using Stein's method of exchangeable pairs where we modify $S$ to a new set $S'$
by removing a random element of $S$ and adding a random element of~$\bS$. Unfortunately,
the expected change $e(S')-e(S)$ depends not just on $e(S)$ but also on the number of
edges between $S$ and~$\bS$, or equivalently, given that $e(G)$ is fixed, on $e(\bS)$.
Thus it will be necessary to track both $e(S)$ and $e(\bS)$. For this we will need
a version of Stein's method that can be applied to the 2-dimensional vector $(e(S),e(\bS))$.

We will use the following $d$-dimensional version of Stein's method which occurs as a special case of a
result of Reinert and R\"ollin \cite[Corollary 3.1]{RR}. A pair of random variables $(\vW,\vW')$
is called \emph{exchangeable} if $(\vW,\vW')$ and $(\vW',\vW)$ have the same distribution.

\begin{theorem}\label{thm:Stein}
 Suppose that\/ $(\vW,\vW')$ is an exchangeable pair of\/ $\R^d$-valued random column vectors
 $\vW=(W_1,\dots,W_d)^T$ and\/ $\vW'=(W'_1,\dots,W'_d)^T$ such that 
 \begin{equation}\label{eq:steinc1}
  \E[\vW]=0\quad\text{and}\quad \E[\vW\vW^T]=\Sigma, 
 \end{equation} 
 where $\Sigma$ is a symmetric positive definite $d\times d$ real matrix.
 Suppose further that
 \begin{equation}\label{eq:SteinW'-W}
  \E\big[\vW'-\vW\mid\vW\big]=-\Lambda \vW 
 \end{equation}
 for an invertible $d\times d$ real matrix $\Lambda$. 
 For $i=1, \dots, d$, let\/ $\lambda^{(i)}=\sum_{m=1}^d |(\Sigma^{-1/2}\Lambda^{-1}\Sigma^{1/2})_{m,i}|$,
 \begin{align*}
  A&=\sum_{i,j} \lambda^{(i)}
  \Big(\Var\,\E\Big[\sum_{k,\ell}\Sigma_{ik}^{-1/2}\Sigma_{j\ell}^{-1/2}(W'_k-W_k)(W'_\ell-W_\ell)\Big|\vW\Big]\Big)^{1/2},\\
  B&=\sum_{i,j,k} \lambda^{(i)}\,\E\,\Big|\sum_{r,s,t}\Sigma_{ir}^{-1/2}\Sigma_{js}^{-1/2}\Sigma_{kt}^{-1/2}
  (W'_r-W_r)(W'_s-W_s)(W'_t-W_t)\Big|,\text{ and}\\
  T&=\tfrac{1}{4d}\Big(\tfrac{A}{2} + \sqrt{\sqrt d B+\tfrac{A^2}{4}}\Big)^2.
 \end{align*}
 Then, with $Z'\sim \Nor(0,1)$ a standard normal random variable, we have 
 \[
  \big|\Prb(W_1\le \sqrt{\Sigma_{11}} z)-\Prb(Z'\le z)\big|
  \le\gamma^2\big(-\tfrac{A}{2}\log T+\tfrac{B}{2\sqrt{T}}+2\sqrt{dT}\big),
 \]
 where $\gamma=\gamma(d)$ is a function of\/ $d$ only.
\end{theorem}

\begin{proof}
To deduce this from \cite[Corollary 3.1]{RR}, and following the notation used there, we take the test function
\[
 h(x):=\one_{(-\infty,\sqrt{\Sigma_{11}}z] \times \R^{d-1}}(x),
\]
which is the indicator function of a convex set in~$\R^d$. As discussed in~\cite{RR} (see also~\cite{BG93}),
this $h$ satisfies the criteria for non-smooth test functions with constant $a \le 2\sqrt{d}$. Also, we take
$R=0$ and hence $C'=0$, $D'=A/2$. Letting $\vZ=(Z_1,Z_2)$ be a standard bivariate and $Z'$ a standard univariate normal
random variable, the supremum in \cite[Corollary 3.1]{RR} is an upper bound for the expression
\[
 \E\big[h(\vW)-h(\Sigma^{1/2}\vZ)\big]=\Prb\big(W_1\le\sqrt{\Sigma_{11}}z\big)-\Prb\big(Z'\le z\big),
\]
as $(\Sigma^{1/2}\vZ)_1\sim \sqrt{\Sigma_{11}}Z'$.
\end{proof}

In our case, we will apply \cref{thm:Stein} with $d=2$, and $\vW$ being a normalised version of $(e(S),e(\bS))$.
Note that $\E_S(e(S))=KM/N$ and $\E_S(e(\bS))=\bK M/N$, where $\bK=\binom{\bk}{2}$ and the expectation is over
the choice of~$S$. We define two random variables $W$ and $\bW$ by
\[
 W:=\frac{e(S)-KM/N}{n},\qquad\bW:=\frac{e(\bS)-\bK M/N}{n},
\]
so that $\E[W]=\E[\bW]=0$. Now pick a uniformly chosen vertex $x$ of $S$ and an independent uniformly chosen
vertex $\bx$ of~$\bS$. Swap the vertices $x$ and $\bx$ to produce two new subsets $S':=S\setminus\{x\}\cup\{\bx\}$
and $\bS':=\bS\setminus\{\bx\}\cup\{x\}$ (see \cref{fig:notation}) and define two more random variables $W'$
and $\bW'$ analogously using the subsets $S'$ and~$\bS'$. We note that the random variables $(W,\bW)$ and
$(W',\bW')$ are exchangeable as $S'$ is uniformly distributed over $k$-sets.
Write $\vW$ for the column vector $(W,\bW)^T$ and similarly for~$\vW'$.

\begin{figure}
\centering
\begin{tikzpicture}[scale=0.5]
\tikzstyle{every node}=[draw,shape=circle,minimum size=3.5,inner sep=0]
\draw[rounded corners=7](0,0) -- ++(6,0) -- ++(0,4) -- ++(-6,0) -- cycle;
\draw[rounded corners=7,fill=blue!5!white](0,0) -- ++(2.5,0) -- ++(0,3) -- ++(1,0) -- ++(0,1) -- ++(-3.5,0) -- cycle;
\draw[rounded corners=7](10,0) -- ++(6,0) -- ++(0,4) -- ++(-6,0) -- cycle;
\draw[rounded corners=7,fill=blue!5!white](10,0) -- ++(3.5,0) -- ++(0,1) -- ++(-1,0) -- ++(0,3) -- ++(-2.5,0) -- cycle;
\node[draw=none,fill=none] at (1.2,2) {$S$};
\node[draw=none,fill=none] at (4.5,2) {$\bS$};
\node[draw=none,fill=none] at (2.45,3.5) {$x$};
\node[draw=none,fill=none] at (3.55,0.55) {$\bx$};
\node[fill=black] at (3,3.5) {};
\node[fill=black] at (3,0.5) {};
\node[draw=none,fill=none] at (11.2,2) {$S'$};
\node[draw=none,fill=none] at (14.5,2) {$\bS'$};
\node[draw=none,fill=none] at (13.55,3.5) {$x$};
\node[draw=none,fill=none] at (12.45,0.55) {$\bx$};
\node[fill=black] at (13,3.5) {};
\node[fill=black] at (13,0.5) {};
\node[draw=none,fill=none] at (8,2) {$\longleftrightarrow$};
\end{tikzpicture}
\caption{Vertex interchange in $G$.}
\label{fig:notation}
\end{figure}
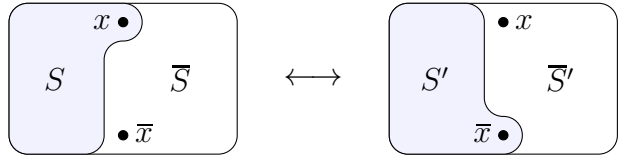

In light of \cref{thm:Stein}, the bulk of the work in this section consists of verifying conditions~\eqref{eq:steinc1}
and \eqref{eq:SteinW'-W}, which we do in \cref{l:eww} and \cref{l:ew} respectively. We will then also show that
$A$ and $B$ are $O(n^{-1/2+\eps})$,
from which a bound of $O(n^{-1/4+\eps})$ will follow. While some of these lemmas are specialised to the random graph
setting in which we will apply them, the first few hold in greater generality.

We begin by verifying~\eqref{eq:SteinW'-W}.  Given $S$, we perform a random exchange of vertices to obtain~$S'$.
The expectation $\E[e(S')-e(S)\mid S]$ is then a linear function of $e(S)$ and $e(S,\bS)$ which,
given that $e(S)+e(\bS)+e(S,\bS)=M$, is also a linear function of $e(S)$ and $e(\bS)$.

\begin{lemma}\label{l:ew}
 For any graph $G$ and\/ $k,\bk>0$ we have
 \[
  \E\big[\vW'-\vW\mid\vW\big]=-\Lambda \vW,\quad\text{where}\qquad
  \Lambda=\frac{1}{k\bk}\begin{pmatrix}n+\bk-1&k-1\\\bk-1&n+k-1\end{pmatrix}.
 \]
\end{lemma}
\begin{proof}
We fix $S$ and take expectations over the choice of $(x,\bx)$. Note that
\begin{equation}\label{eq:de}
 e(S')-e(S)=d_S(\bx)-d_S(x)-\one_{x\bx},
\end{equation}
where $d_S(y)$ is the number of edges from $y$ to $S$ and $\one_{xy}$ represents an indicator function of the
event that $xy\in E(G)$. Now fixing $G$ and $S$ and taking expectations over $(x,\bx)$ we have
\begin{align*}
 \E\big[e(S')-e(S)\mid S\big]&=\E_{\bx}d_S(\bx)-\E_x d_S(x)-\E_{x,\bx}\one_{x\bx}\\
 &=\tfrac{e(S,\bS)}{\bk}-\tfrac{2e(S)}{k}-\tfrac{e(S,\bS)}{k\bk}\\
 &=\tfrac{1}{k\bk}\big((k-1)e(S,\bS)-2\bk e(S)\big)\\
 &=\tfrac{1}{k\bk}\big((k-1)M-(k-1+2\bk)e(S)-(k-1)e(\bS)\big)\\
 &=\tfrac{1}{k\bk}\big((k-1)M-(n-1+\bk)KM/N-(k-1)\bK M/N\big)\\
 &\quad-\tfrac{1}{k\bk}\big((n-1+\bk)nW+(k-1)n\bW\big)\\
 &=-\tfrac{1}{k\bk}\big((n-1+\bk)nW+(k-1)n\bW\big),
\end{align*}
where we have used that $k+\bk=n$ and $\E[e(S')-e(S)]=\E[W]=\E[\bW]=0$, so that the constant term must vanish.
Hence (as $\vW$ is $S$-measurable and the expectation only depends on~$\vW$)
\[
 \E\big[W'-W\mid \vW\big]=-\tfrac{n-1+\bk}{k\bk}W-\tfrac{k-1}{k\bk}\bW.
\]
By symmetry, $\E[\bW'-\bW\mid \vW]=-\frac{n-1+k}{k\bk}\bW-\frac{\bk-1}{k\bk}W$ and the result follows.
\end{proof}

We note for future reference that
\begin{equation}\label{eq:Li}
 \Lambda^{-1}=\frac{k\bk}{2n(n-1)}\begin{pmatrix}n+k-1&1-k\\1-\bk&n+\bk-1\end{pmatrix}
\end{equation}
and the entries of $\Lambda^{-1}$ are all~$O(n)$.

We now turn to verifying~\eqref{eq:steinc1}. We have already noted that $\E \vW=0$,
and must check that $\vW$ has a positive definite covariance matrix $\Sigma$. 

Let $P$ be the number of copies of the path on three vertices in~$G$, so that
\[
 P=\sum_{v\in V(G)} \binom{d(v)}{2}.
\]
Recall that $d(G)=2M/n$ is the average degree of $G$ and let $V$ be $n$ times the variance of the degrees of~$G$, so that
\[
 V=\sum_{v\in V(G)} d(v)^2-nd(G)^2.
\]
We note that
\begin{equation}\label{eq:V}
 2Pn=Vn+4M^2-2Mn.
\end{equation}

\begin{lemma}\label{l:eww}
 For any graph~$G$ and\/ $k,\bk\ge2$, we have $\E[\vW\vW^T]=\Sigma$, where
 \[
  \Sigma:=\frac{2K\bK}{n^2N^2(n-2)(n-3)}
 \begin{pmatrix}M(N-M)+\tfrac{k-2}{\bk-1}NV&M(N-M)-NV\\
  M(N-M)-NV&M(N-M)+\tfrac{\bk-2}{k-1}NV\end{pmatrix}.
 \]
 In particular, $\Sigma$ is symmetric and positive definite if and only if\/ $V>0$.
\end{lemma}
\begin{proof}
We first note that
\[
 \E\big[W^2\big]=\frac{1}{n^2}\Var(e(S))
 =\frac{1}{n^2}\sum_{e,f\in E(G)}\left(\Prb(e,f\subseteq S)-\Prb(e\subseteq S)\Prb(f\subseteq S)\right),
\]
where the sum is over ordered pairs $(e,f)$ of edges of $G$. We can classify these pairs
according to the size of $e\cap f$. Namely, there are $M$ pairs where $e=f$, $2P$ pairs where $|e\cap f|=1$
and $M^2-M-2P$ pairs where $e\cap f=\emptyset$. From this we deduce that
\[
 \E\big[W^2\big]=\frac{1}{n^2}\bigg(\frac{(k)_2}{(n)_2}M
 +\frac{(k)_3}{(n)_3}2P+\frac{(k)_4}{(n)_4}(M^2-M-2P)
 -\frac{(k)_2^2}{(n)_2^2}M^2\bigg),
\]
where $(x)_{r}=x(x-1)\dots(x-r+1)$ denotes the falling factorial. Note that $k,\bk\ge2$ implies $n\ge4$, so $(n)_4>0$.
Using \eqref{eq:V} to rewrite $2P$ in terms of $M$, $n$ and $V$ and determining the coefficients of the
terms containing $M$, $M^2$ and~$V$, respectively, gives
\[
 \E\big[W^2\big]=\frac{2\bk K}{n^2N(n)_4}\big(M(N-M)(\bk-1)+(k-2)NV\big),
\]
as well as the symmetric expression
\[
 \E\big[\bW^2\big]=\frac{2k\bK}{n^2N(n)_4}\big(M(N-M)(k-1)+(\bk-2)NV\big).
\]
A similar calculation also gives
\begin{align*}
 \E\big[W\bW\big]
 &=\frac{1}{n^2}\sum_{e,f\in E(G)}\left(\Prb(e\subseteq S,f\subseteq \bS)-\Prb(e\subseteq S)\Prb(f\subseteq\bS)\right)\\
 &=\frac{1}{n^2}\bigg(\frac{(k)_2(\bk)_2}{(n)_4}(M^2-M-2P)-\frac{(k)_2(\bk)_2}{(n)_2^2}M^2\bigg)\\
 &=\frac{4K\bK}{n^2 N(n)_4}\big(M(N-M)-NV\big).
\end{align*}
The expression for $\Sigma$ then follows from these expressions. We note that for any vector
$v=(\alpha,\beta)^T$, $v^T\Sigma v$ is the variance of the random variable $\alpha W+\beta \bW$.
Thus if $\Sigma$ is not positive definite, then there must be a non-trivial linear relationship
$\alpha e(S)+\beta e(\bS)=c$ for all choices of~$S$. Assume first that $\alpha\ne-\beta$. Then by interchanging
$x\in S$ with $\bx\in \bS$ we deduce from \eqref{eq:de} that
\[
 \alpha\big(d_S(x)-d_S(\bx)-\one_{x\bx}\big)+\beta\big(d_{\bS}(\bx)-d_{\bS}(x)-\one_{x\bx}\big)=0.
\]
Thus $\alpha d_S(x)-\beta d_{\bS}(x)$ is a constant ($=\alpha d_S(\bx)-\beta d_S(\bx)$) for all
$x\in S\setminus N(\bx)$, and a different constant ($=\alpha d_S(\bx)-\beta d_S(\bx)+(\alpha+\beta)$) for all
$x\in S\cap N(\bx)$. As this holds for all choices of $\bx\in \bS$ we deduce that $S\cap N(\bx)$ is
independent of $\bx\in\bS$. But by varying~$S$ (and assuming $k,\bk\ge2$) 
we deduce that $N(\bx)\setminus\{\bx,\by\}=N(\by)\setminus\{\bx,\by\}$ for any two vertices
$\bx,\by\in V(G)$. This is a contradiction
unless $G$ is either the empty graph or the complete graph. On the other hand, $v^T\Sigma v=0$ for $v=(1,-1)^T$ implies
that $d(x)=d_S(x)+d_{\bS}(x)$ is constant for all $x\in S$ which, as this holds for all $S$ and $k\ge2$,
immediately implies that $G$ is a regular graph. Thus in all cases $G$ is regular, which implies $V=0$.
Conversely, if $V=0$ then $\Sigma$ is clearly singular.
\end{proof}

In the case that $V=0$, so $G$ is regular, $W=\bW$ and \cref{l:eww} reduces
to the statement that $\E[W^2]=\frac{2K\bK}{n^2N^2(n-2)(n-3)}M(N-M)$.

We now move on to bounds for the terms $A$ and $B$ appearing in \cref{thm:Stein}. We start with the following
lemma which implies a suitable bound on~$B$.

\begin{lemma}\label{l:SteinB}
 Suppose that, for a constant\/ $\eps>0$, $G$ satisfies 
 \begin{equation}\label{eq:3}
  \frac{1}{n}\sum_{v\in V(G)} |d(v)-d(G)|^3=O(n^{3/2+ \eps}).
 \end{equation} 
 Then $\E\big[|W'-W|^i|\bW'-\bW|^{3-i}\big]=O(n^{-3/2+\eps})$ for $i=0,\dots,3$.
\end{lemma}
\begin{remark}
Indeed, for any $\eps>0$, condition~\eqref{eq:3} holds w.v.h.p.\ in $G_{n,p}$ and hence also in $G_{n,M}$ because,
for a fixed vertex $v$, w.v.h.p.\ we have $|d(v)-d(G)| \le n^{1/2+\eps/3}$, and thus by a union
bound this holds w.v.h.p.\ for all vertices.
\end{remark}

\begin{proof}
It is enough that the third moments $\E[|e(S')-e(S)|^3]$ and $\E[|e(\bS')-e(\bS)|^3]$ are both $O(n^{3/2+\eps})$
as $W'-W=(e(S')-e(S))/n$, $\bW'-\bW=(e(\bS')-e(\bS))/n$, and in general $\E|XY^2|\le (\E|X^3|)^{1/3}(\E|Y^3|)^{2/3}$
and $\E|X^2Y|\le (\E|X^3|)^{2/3}(\E|Y^3|)^{1/3}$ by H\"older's inequality. We shall only show that
$\E[|e(S')-e(S)|^3]=O(n^{3/2+\eps})$ as the result for $\E[|e(\bS')-e(\bS)|^3]$ follows by symmetry.

Using \eqref{eq:de}, that $\E[(|X|+|Y|)^3]=O(\E|X|^3+\E|Y|^3)$ for random variables $X$, $Y$, and $d(G)\le n-1$, we get
\begin{align*}
 \E\big[|e(S')-e(S)|^3\big]
 &=\E\big[|d_S(\bx)-d_S(x)-\one_{x\bx}|^3\big]\\
 &=O\big(\E\big[|d_S(\bx)-\tfrac{k}{n-1}d(G)|^3\big]+\E\big[|d_S(x)-\tfrac{k-1}{n-1}d(G)|^3\big]+1\big).
\end{align*}
Thus we need to show that $\E[|d_S(\bx)-\tfrac{k}{n-1}d(G)|^3]$ and $\E[|d_S(x)-\tfrac{k-1}{n-1}d(G)|^3]$ are $O(n^{3/2+\eps})$.
Fixing $\bx$ and taking $S$ to be a random $k$-subset of $V(G)\setminus\{\bx\}$ gives
\[
 \E\big[|d_S(\bx)-\tfrac{k}{n-1}d(G)|^3\mid \bx\big]=O\big(\E\big[|X-\E[X]|^3\big]+\big|\E[X]-\tfrac{k}{n-1}d(G)\big|^3\big),
\]
where $X$ is a hypergeometric random variable obtained by selecting $k$ items from a set $V(G)\setminus\{\bx\}$ of size $n-1$
and counting the number of these which occur in a given set $N(\bx)$ of size $d(\bx)$. We note that $\E[|X-\E X|^3]=O(n^{3/2})$
and $\E X=\frac{k}{n-1}d(\bx)$. Then randomising over~$\bx$, and using $k\le n$ and~\eqref{eq:3}, we see that 
\begin{align*}
 \E\big[|d_S(\bx)-\tfrac{k}{n-1}d(G)|^3\big]
 &=\tfrac{1}{n}\sum_{\bx} \E\big[|d_S(\bx)-\tfrac{k}{n-1}d(G)|^3\mid \bx\big]\\
 &=O\Big(n^{3/2}+\tfrac{1}{n}\sum_{\bx}|d(\bx)-d(G)|^3\Big)=O\big(n^{3/2+\eps}\big).
\end{align*}
The proof is completed by an analogous argument for $\E[|d_S(x)-\tfrac{k-1}{n-1}d(G)|^3]$.
\end{proof}

The last missing ingredient for our application of \cref{thm:Stein} is a suitable bound on the $A$ appearing there,
which will be shown in \cref{l:var} below. First, we need a concentration estimate for the variance (and covariance)
of various degrees in~$G$.

\begin{lemma}\label{lem:Vconc}
 Let\/ $p=p(n)\in[\delta,1-\delta]$ and fix a set\/ $S$ of size~$k$.
 Let\/ $x$ be a random element of\/ $S$ and\/ $\bx$ a random element of\/ $\bS$ as before.
 Then, for any $\eps>0$, w.v.h.p.\ in $G_{n,p}$ we have
 \begin{equation}\label{eq:varx}
  \big|\Var_x d_S(x)-p(1-p)k\big|=O(n^{1/2+\eps}),
 \end{equation}
 \begin{equation}
  \big|\Var_{\bx} d_S(\bx)-p(1-p)k\big|=O(n^{1/2+\eps}),
 \end{equation}
 and
 \begin{equation}
  \Cov_{\bx}\big(d_S(\bx),d_{\bS}(\bx)\big)
  =\E_{\bx}\big[d_S(\bx)d_{\bS}(\bx)\big]-\E_{\bx}\big[d_S(\bx)\big]\E_{\bx}\big[d_{\bS}(\bx)\big]
  = O(n^{1/2+\eps}).
 \end{equation}
\end{lemma}
\begin{proof}
In the following sums, the indices $v$, $w$, $y$, $z$ run over vertices in~$S$, and the indices $\bv$, $\bw$, $\by$, $\bz$
run over vertices in~$\bS$. We use primes on sums to indicate that summation is over tuples of distinct elements only. 
Let $\one_{vw}$ be the indicator variable that the edge $vw$ is present in $G_{n,p}$ as before;
in particular we have $\one_{vv}=0$ for any vertex~$v$. We may then expand
\begin{align*}
 V_1&:=k^2 \cdot \Var_x d_S(x)=k\sum_v d_S(v)^2-\sum_{v,w}d_S(v)d_S(w)\\
 &\phantom{:}=k\sum_{v,y,z}\one_{vy}\one_{vz}-\sum_{v,w,y,z}\one_{vy}\one_{wz}\\
 &\phantom{:}=(k-2)\sideset{}{'}\sum_{v,y}\one_{vy}
  +(k-4)\sideset{}{'}\sum_{v,y,z}\one_{vy}\one_{vz}
  -\sideset{}{'}\sum_{v,w,y,z}\one_{vy}\one_{wz}.
\intertext{Similarly,}
 V_2&:=\bk^2\cdot \Var_{\bx}d_S(\bx)=\bk\sum_{\bv} d_S(\bv)^2-\sum_{\bv,\bw}d_S(\bv)d_S(\bw)\\
 &\phantom{:}=\bk\sum_{\bv,y,z}\one_{\bv y}\one_{\bv z}-\sum_{\bv,\bw,y,z}\one_{\bv y}\one_{\bw z}\\
 &\phantom{:}=(\bk-1)\sum_{\bv} \sum_y \one_{\bv y}
  +(\bk-1)\sum_{\bv}\sideset{}{'}\sum_{y,z}\one_{\bv y}\one_{\bv z}
  -\sideset{}{'}\sum_{\bv,\bw}\sum_{y,z}\one_{\bv y}\one_{\bw z},
\intertext{and}
 V_3&:=\bk^2 \cdot \Cov_{\bx}(d_S(\bx),d_{\bS}(\bx))
  =\bk\sum_{\bv}d_S(\bv)d_{\bS}(\bv)-\sum_{\bv,\bw}d_S(\bv)d_{\bS}(\bw)\\
 &\phantom{:}=\bk\sum_{\bv,\by,y}\one_{\bv y}\one_{\bv\,\by}-\sum_{\bv,\bw,\by,y}\one_{\bv y}\one_{\bw\,\by}\\
 &\phantom{:}=(\bk-2)\sideset{}{'}\sum_{\bv,\by}\sum_y\one_{\bv y}\one_{\bv\,\by}
  -\sideset{}{'}\sum_{\bv,\bw,\by}\sum_y\one_{\bv y}\one_{\bw\,\by}.
\end{align*}
This allows us to compute
\begin{align*}
 \E_G[V_1]&=(k-2)\cdot (k)_2\,p+(k-4)\cdot (k)_3\,p^2-1\cdot (k)_4\,p^2=(k)_3\,p(1-p),\\
 \E_G[V_2]&=(\bk-1)\cdot \bk kp + (\bk-1)\cdot \bk (k)_2\,p^2 -1\cdot(\bk)_2 k^2 p^2=(\bk)_2 k p (1-p),\\
 \E_G[V_3]&=(\bk-2)\cdot(\bk)_2kp^2-1\cdot(\bk)_3kp^2=0.
\end{align*}
Now for the concentration inequalities, enumerate the edges of $K_n$ as $e_1,\dots,e_{\binom{n}{2}}$
and let $\cF_s$ be the $\sigma$-algebra generated by $\one_{e_1},\dots,\one_{e_s}$. For $i=1,2,3$, write
\[
 X_s^{(i)}=\E[V_i\mid \cF_s]
\]
to denote the standard edge exposure martingale for~$V_i$. We show that for $i=1,2,3$,
w.v.h.p.\ $|X^{(i)}_s-X^{(i)}_{s-1}|= O(n^{3/2+\eps})$ for all~$s$.
Suppose we expose the edge $e_s=ab$ at step~$s$. If we write $V_1=c_{ab}\one_{ab}+c$,
where $c_{ab}$ and $c$ do not depend on~$\one_{ab}$, then the coefficient $c_{ab}$ is given by
\begin{align*}
 c_{ab}&=(k-2)\cdot 2+(k-4)\cdot2\sum_{y\ne a,b}(\one_{ay}+\one_{by})-4\sideset{}{'}\sum_{y,z\ne a,b}\one_{yz}\\
 &=2k d_{S}(a)+2k d_S(b)-4\sum_{y} d_S(y)+O(n).
\end{align*}
In the last step, to obtain a simpler expression, we added some terms of order $O(n)$ that do
depend on~$\one_{ab}$, but that does not change the fact that $c_{ab}$ does not. We now get
\begin{align*}
 \big|X_s^{(1)}-X_{s-1}^{(1)}\big|&=\big(\one_{ab}-p\big) \E[c_{ab}\mid \cF_s]\\
 &=2\big(\one_{ab}-p\big)\Big(k\E[d_S(a)\,|\,\cF_s]+k\E[d_S(b)\,|\,\cF_s]-2\sum_y \E[d_S(y)\,|\,\cF_s]\Big)+O(n).
\end{align*}
The coefficient of $\one_{ab}$ in $V_2$ is 0 if $a$ and $b$ are both in~$S$, or both in~$\bS$.
Otherwise we may assume $a \in S$, $b \in \bS$, and then the coefficient is
\[
 c_{ab}=(\bk-1)+(\bk-1)\cdot 2\sum_{y\ne a}\one_{by}-1\cdot 2\sum_{y\ne a,\by\ne b}\one_{\by y}
 =2\bk d_S(b)-2\sum_{\by} d_S(\by)+O(n),
\]
yielding
\[
 \big|X_s^{(2)}-X_{s-1}^{(2)}\big|=2\big(\one_{ab}-p\big)
 \Big(\bk\,\E[d_S(b)\mid\cF_s]-\sum_{\by} \E[d_S({\by})\mid\cF_s]\Big)+O(n). 
\]
The coefficient of $\one_{ab}$ in $V_3$ is also 0 if $a$ and $b$ are both in~$S$. When $a,b\in\bS$, the coefficient is
\[
 c_{ab}=(\bk-2)\sum_y (\one_{a y}+\one_{b y})-2\sum_{\by \ne a,b} \sum_{y}\one_{\by y}
 =\bk d_S(a)+\bk d_S(b)-2\sum_{\by}d_S(\by)+O(n),
\]
which gives 
\[
 \big|X_s^{(3)}-X_{s-1}^{(3)}\big|=\big(\one_{ab}-p\big)
 \Big(\bk\,\E[d_S(a)+d_S(b)\mid\cF_s]-2\sum_{\by} \E[d_S(\by)\mid\cF_s]\Big)+O(n).
\]
When $a \in S$, $b \in \bS$, we have
\[
 c_{ab}=(\bk-2)\sum_{\bz}\one_{b \bz} - \sideset{}{'}\sum_{\bw,\bz \ne b}\one_{\bw\,\bz}
 =\bk d_{\bS}(b)-\sum_{\bw}d_{\bS}(\bw)+O(n),
\]
and hence 
\[
 \big|X_s^{(3)}-X_{s-1}^{(3)}\big|=\big(\one_{ab}-p\big)  
 \Big(\bk\,\E[d_{\bS}(b)\mid\cF_s]-\sum_{\bw} \E[d_{\bS}(\bw)\mid\cF_s]\Big)+O(n).
\]
Applying the Chernoff bound to the exposed edges, we deduce that for a fixed $s$ and $v \in V(G)$, \wvhp, $\E[d_S(v)\mid\cF_s]=pk+O(n^{1/2+\eps/2})$. Hence, by a union bound, we have that
w.v.h.p.\ this holds for all $v$ and all~$s$. Similarly, we can show that $\E[d_{\bS}(v)\mid\cF_s]=p\bk+O(n^{1/2+\eps/2})$
for all $v$ and~$s$. This can now be applied to all the expressions derived above to see that, for $i=1,2,3$,
we have w.v.h.p.\ $|X^{(i)}_s-X^{(i)}_{s-1}|=O(n^{3/2+\eps/2})$ for all~$s$. The Azuma--Hoeffding inequality (\cref{AH}) with
$t=n^{5/2+\eps}$, $c_s=O(n^{3/2+\eps/2})$ and hence $\sigma^2=O(n^{5+\eps})$ now gives the desired bound
\[
 \Prb_G\big(|V_i-\E[V_i]|>t\big)\le\exp(-n^{\Omega(1)})+2\exp\big(-t^2/2\sigma^2\big)=\exp\big(-n^{\Omega(1)}\big)
\] 
for $i=1,2,3$. The result follows as $k,\bk=\Theta(n)$.
\end{proof}

\begin{remark}\label{rem:vconc}
The proof of the first part of \cref{lem:Vconc} (that is, equation~\eqref{eq:varx}), does not require that $\bk\ge \delta n$
and hence the result also holds for $S=V(G)$. This gives precisely that $V=p(1-p)n^2+O(n^{3/2+\eps})$, which we will use
in the following subsection.
\end{remark}

\begin{lemma}\label{l:var}
 In\/ $G=G_{n,M}$, for $i=0,1,2$ and all\/ $\eps>0$, w.v.h.p.\ we have
 \[
  \Var\big(\E[(W'-W)^i(\bW'-\bW)^{2-i}\mid \vW]\big)=O(n^{-3+\eps}).
 \]
\end{lemma}
\begin{proof}
Let $0<\eps'<\eps/2$. We work in $G_{n,p}$ with $p=M/\binom{n}{2}$ and start with a fixed set~$S$. We then have
\begin{align*}
 \E_{x,\bx}\big[(e(S')-e(S))^2\big]
 &=\E_{x,\bx}\big[(d_S(\bx)-d_S(x)-\one_{x\bx})^2\big]\\
 &=\E_{x,\bx}\big[(d_S(\bx)-d_S(x))^2\big]+\E_{x,\bx}\big[(1+2d_S(\bx)-2d_S(x))\one_{x\bx}\big]\\
 &=\Var_{\bx}d_S(\bx)+\Var_x d_S(x)+\big(\E_{\bx} d_S(\bx)-\E_{x} d_S(x)\big)^2\\
 &\quad+\E_{x,\bx}\big[(1+2d_S(\bx)-2d_S(x))\one_{x\bx}\big],
\end{align*}
where we have used the independence of $x$ and~$\bx$.

By \cref{lem:Vconc}, w.v.h.p.\ in $G=G_{n,p}$, $\Var_{\bx} d_S(\bx)$ and $\Var_x d_S(x)$ are both $p(1-p)k+O(n^{1/2+\eps'})$.
We also have w.v.h.p.\ that $e(S,S')-pk\bk$ and $e(S)-p\binom{k}{2}$ are $O(n^{1+\eps'})$ and thus
\begin{align*}
 \E_{\bx} d_S(\bx)-\E_{x} d_S(x)=\tfrac{e(S,\bS)}{\bk}-\tfrac{2e(S)}{k}=pk-p(k-1)+O(n^{\eps'})=O(n^{\eps'}).
\end{align*}
Finally,  w.v.h.p.\ for all $y \in S$ and $\by \in \bS$, we have $d_S(\by), d_S(y)=pk+O(n^{1/2+\eps'})$. Thus
\[
 \E_{x,\bx}\big[(1+2d_S(\bx)-2d_S(x))\one_{x\bx}\big]=O(n^{1/2+\eps'}).
\]
Putting all of this together, we see that for fixed $S$ w.v.h.p.\ in $G$ we have
\begin{align}\label{eq:square_conc}
 \big|\E_{x,\bx}\big[(e(S')-e(S))^2\big]-2p(1-p)k\big| \le \tfrac{1}{2} n^{1/2+\eps/2}.
\end{align} 
With $S$ no longer fixed and denoting by $X$ the number of $S$ failing~\eqref{eq:square_conc},
we have $\E[X]\le \exp(-2n^{\eta})\binom{n}{k}$ for some constant $\eta>0$ and hence Markov gives
\[
 \Prb_G\big(X>\exp(-n^{\eta})\tbinom{n}{k}\big)
 \le \frac{\E[X]}{\exp(-n^{\eta})\binom{n}{k}}\le\exp(-n^{\eta}),
\]
so w.v.h.p.\ in $G$ there is at most an $\exp(-n^{\eta})$ proportion of choices of $S$ not satisfying~\eqref{eq:square_conc}.

This property holds w.v.h.p.\ in $G_{n,p}$ and thus in $G_{n,M}$ and we now fix $G$ with this property.
We want to condition on $\vW$ and let $\eta'<\eta$. Say that $\vw$ is {\em typical} if $\Prb_S(\vW=\vw)\ge\exp(-n^{\eta'})$.
For a typical~$\vw$, since $|(e(S')-e(S))^2|$ is bounded by~$n^4$, we have
\[
 \big|\E\big[(e(S')-e(S))^2\mid \vW=\vw\big]-2p(1-p)k\big|
 \le \exp(-n^\eta + n^{\eta'})n^4+\tfrac{1}{2}n^{1/2+\eps/2}\le n^{1/2+\eps/2}.
\]
On the other hand, there are at most $\binom{n}{2}^2\le n^4$ possible values of $(e(S),e(\bS))$ and hence at most $n^4$
possible values of~$\bW$. Thus there are at most $n^4 \exp(-n^{\eta'}) \binom{n}{k}$ choices of $S$ where $\vW$
takes a non-typical value, yielding that
\[
 \Prb_S\big(|\E\big[(e(S')-e(S))^2\mid \vW\big]-2p(1-p)k\big|>n^{1/2+\eps/2}\big)\le n^4\exp(-n^{\eta'}).
\]
Additionally, using that $0\le \E[(e(S')-e(S))^2\mid\vW]\le n^4$, we get
\[
 \Var\big(\E\big[(e(S')-e(S))^2\mid \vW\big]\big)\le n^{12}\exp(-n^{\eta'})+n^{1+\eps}=O(n^{1+\eps}).
\]
Now noting that $W'-W=(e(S')-e(S))/n$, we obtain that w.v.h.p.\ in $G_{n,M}$ we have
\[
 \Var\big(\E\big[(W'-W)^2\mid \vW\big]\big)
 = n^{-4}\Var\big(\E\big[(e(S')-e(S))^2\mid S\big]\big)=O(n^{-3+\eps}).
\]
By symmetry, we also have $\Var(\E[(\bW'-\bW)^2\mid \vW])=O(n^{-3+\eps})$.
The proof for $\Var(\E[(W'-W)(\bW'-\bW)\mid \vW])$ is similar. Indeed, for fixed $S$ we can write
\begin{align*}
 \E_{x,\bx}\big[(e(S')-e(S))(e(\bS')-e(\bS))\big]
 &=\E_{x,\bx}\big[(d_S(\bx)-d_S(x)-\one_{x\bx})(d_{\bS}(x)-d_{\bS}(\bx)-\one_{x\bx})\big]\\
 &=-\Cov_{\bx}\big(d_S(\bx), d_{\bS}(\bx)\big)-\Cov_{x}\big(d_S(x), d_{\bS}(x)\big)\\
 &\quad+\big(\E_{\bx}[d_S(\bx)]-\E_x[d_S(x)]\big)\big(\E_x[d_{\bS}(x)]-\E_{\bx}[d_{\bS}(\bx)]\big)\\
 &\quad+\E_{x,\bx}\big[\one_{x\bx}(1-d_S(\bx)+d_S(x)-d_{\bS}(x)+d_{\bS}(\bx))\big].
\end{align*}
Once again applying \cref{lem:Vconc}, we know that w.v.h.p.\ the covariances $\Cov_{\bx}(d_S(\bx),d_{\bS}(\bx))$
and $\Cov_{x}(d_S(x),d_{\bS}(x))$ are $O(n^{1/2+\eps'})$ in $G_{n,p}$. Further, w.v.h.p. $e(S,S')-pk\bk$,
$e(S)-pk^2/2$ and $e(S,\bS)-p\bk^2/2$ are $O(n^{1+\eps'})$ and so
\begin{align*}
 \big(\E_{\bx}[d_S(\bx)]-\E_x[d_S(x)]\big)\big(\E_x[d_{\bS}(x)]&-\E_{\bx}[d_{\bS}(\bx)]\big)\\
 &=\big(e(S,\bS)/\bk-2e(S)/k\big)\big(2e(\bS)/\bk-e(S,\bS)/k\big)\\
 &=O(n^{\eps'})O(n^{\eps'})=O(n^{2\eps'}).
\end{align*}
Finally, as above w.v.h.p.\ we have $\E_{x,\bx}[\one_{x\bx}(1-d_S(\bx)+d_S(x)-d_{\bS}(x)+d_{\bS}(\bx))]=O(n^{1/2+\eps'})$. 
Putting everything together, we obtain $\E_{x,\bx}[(e(S')-e(S))(e(\bS')-e(\bS))]=O(n^{1/2+\eps'})$.
Arguing as above, we can then deduce that w.v.h.p.\ in $G_{n,M}$ we have
\begin{align*}
 \Var\big(\E\big[(W'-W)(\bW'-\bW)\mid \vW\big]\big)
 &=n^{-4}\Var\big(\E\big[(e(S')-e(S))(e(\bS')-e(\bS))\mid \vW\big]\big)\\
 &=O(n^{-3+\eps})
\end{align*}
as desired.
\end{proof}

\begin{lemma}\label{lem:Stein}
 For any $\eps>0$, w.v.h.p.\ in $G_{n,M}$, for any~$z$,
 \[
  \big|\Prb_S(W\le \sqrt{\Sigma_{11}} z)-\Prb_S(Z'\le z)\big|= O(n^{-1/4+\eps}),
 \]
 where $Z'\sim \Nor(0,1)$ is a standard normal random variable,
 \[
  \Sigma_{11}=\Var(W)=\frac{2K\bK}{n^2N^2(n-2)(n-3)}\big(M(N-M)+\tfrac{k-2}{\bk-1}NV\big)
 \]
 and\/ $V=\sum_{i=1}^n d(v)^2-nd(G)^2$.
\end{lemma}
\begin{proof}
We apply \cref{thm:Stein} to our $(\vW,\vW')$. \cref{l:ew} shows that \eqref{eq:SteinW'-W} holds for the
$\Lambda$ given there and \cref{l:eww} shows that \eqref{eq:steinc1} holds with the $\Sigma$ given there.
We note that under the conditions of the theorem, the entries of both $\Sigma^{1/2}$ and $\Sigma^{-1/2}$ are $O(1)$.
Indeed, we have $V=p(1-p)n^2+O(n^{3/2+\eps})$ where $p=M/N$ by \cref{rem:vconc}, and hence $NV=(2+o(1))M(N-M)$. Thus
\[
 \Sigma=\Theta(1)\cdot\begin{pmatrix}1+2\kappa+o(1)&-1+o(1)\\-1+o(1)&1+2\kappa^{-1}+o(1)\end{pmatrix}
\]
where $\kappa=k/\bk=\Theta(1)$.
The entries of $\Lambda^{-1}$ are $O(n)$ by~\eqref{eq:Li}, and hence the terms $\lambda^{(i)}$ of \cref{thm:Stein} are also~$O(n)$.
Since the entries of $\Sigma^{-1/2}$ are $O(1)$, \cref{l:SteinB} gives $B=O(n^{-1/2+ \eps})$ and \cref{l:var} gives
$A=O(n^{-1/2+\eps})$. This then implies that $T=O(n^{-1/2+\eps})$ and $B/\sqrt{T}=O(B/\sqrt{B})=O(n^{-1/4+\eps})$.
Hence our application of \cref{thm:Stein} yields, for any~$z$,
 \[
  \big|\Prb_S(W\le \sqrt{\Sigma_{11}} z)-\Prb_S(Z'\le z)\big|= O(n^{-1/4+\eps}),
 \]
as desired.
\end{proof}

\subsection{Proof of Theorem~\ref{thm:clt}}
\label{subsec:clt}

We will now show that the output from Stein's method, namely \cref{lem:Stein}, can be reformulated into the statement of
\cref{thm:clt}. First, note that the $V$ and by extension the $\Sigma_{11}$ appearing in \cref{thm:clt}
depend on the graph $G=G_{n,M}$. To get around this dependency, we can use the fact that $V$ is well-concentrated
in $G_{n,M}$ which was essentially established in Remark~\ref{rem:vconc}. The following calculation is the source of
the definition $\lambda=\frac{(n^2-k^2)k^2}{2n^4}\cdot\frac{M(N-M)}{N^2}$ that we made in~\eqref{eq:lambda}. 

\begin{corollary}\label{cor:sigma11}
 For any $\eps>0$, \wvhp\ we have $\Sigma_{11}=\lambda(1+O(n^{-1/2+\eps}))$.
\end{corollary}
\begin{proof}
By \cref{lem:Vconc} (or specifically Remark~\ref{rem:vconc}) we may assume $V=p(1-p)n^2+O(n^{3/2+\eps})$ where $p=M/N$
is bounded away from 0 and~1. Now as $k,\bk=\Theta(n)$ and $\bk=n-k$,
\begin{align*}
 \Sigma_{11}&=\tfrac{2K\bK}{n^2N^2(n-2)(n-3)}\big(M(N-M)+\tfrac{k-2}{\bk-1}NV\big)\\
 &=\tfrac{k^2\bk^2}{2n^4}p(1-p)\big(1+\tfrac{2k}{\bk}+O(n^{-1/2+\eps})\big)\\
 &=\tfrac{k^2}{2n^4}p(1-p)\bk(\bk+2k)\big(1+O(n^{-1/2+\eps})\big)\\
 &=\tfrac{k^2}{2n^4}\cdot\tfrac{M(N-M)}{N^2}\cdot (n^2-k^2)\big(1+O(n^{-1/2+\eps})\big)\\
 &=\lambda\big(1+O(n^{-1/2+\eps})\big).\qedhere
\end{align*}
\end{proof}

To obtain our modified statement in \cref{thm:clt}, in addition to getting rid of dependency on the graph,
we also transfer to an expression given in terms of the edge count $e(S)$ without normalisation.

\begin{proof}[Proof of \cref{thm:clt}]
It is enough to show that the left hand side of~\eqref{eq:clt} is $O(n^{-1/4+\eps})$. As before,
let $Z'$ be a standard normal random variable. Using the shorthand $\rho(x):=(x-KM/N)/\sqrt{\Sigma_{11}}n$ and
$c:=\sqrt{\Sigma_{11}/\lambda}$, we may write
\begin{align*} 
\big|\Prb_S(e(S)\le z)-\Prb_S(Z\le z)\big|
 &= \big|\Prb_S(W\le \sqrt{\Sigma_{11}} \rho(z))-\Prb_S(Z'\le \sqrt{\Sigma_{11}/\lambda}\rho(z))\big|,
\end{align*}
and this expression can be bounded by
\[
 \big|\Prb_S(W\le\sqrt{\Sigma_{11}}\rho(z))-\Prb_S(Z'\le\rho(z))\big|
 +\big|\Prb_S(Z'\le\sqrt{\Sigma_{11}/\lambda}\rho(z))-\Prb_S(Z'\le\rho(z))\big|.
\]
The first of these summands is with very high probability $O(n^{-1/4+\eps})$ by \cref{lem:Stein}. 
As for the second summand, we may assume  $\rho(z)=O(\log n)$ as otherwise $\Prb(Z'\le \rho(z))$ and
$\Prb(Z'\le c\rho(z))$ are both within $O(n^{-1})$ of 0 or~1, depending on the sign on~$\rho(z)$.
Since the probability density function of $Z'$ is bounded by~1 and $c=1+O(n^{-1/2+\eps})$ by \cref{cor:sigma11}, we obtain
\[
 \big|\Prb(Z'\le \rho(z))-\Prb(Z'\le c\rho(z))\big|\le |c\rho(z)-\rho(z)|=O(|c-1|\log n)=O(n^{-1/2+2\eps}),
\]
completing the proof.
\end{proof}

\section{The local limit theorem}\label{sec:llt}

Recall that $S$ is a random subset of the vertex set of $G=G_{n,M}$ of size~$k$, where
$M/N,k/n\in[\delta,1-\delta]$ with $N=\binom{n}{2}$. We also define
$\lambda=\frac{(n^2-k^2)k^2}{2n^4}\cdot \frac{M(N-M)}{N^2}$ and note that $\lambda=\Theta(1)$.
We now derive a local limit theorem for the empirical distribution of the number $e(S)$ of edges induced
by~$S$, where the distribution is over choices of $k$-set $S$ with $G$ fixed, restated below.
 
\llt*

The conclusion of this theorem and further ones later take the form that, for all values $z\in\{0,\dots,\binom{n}{2}\}$
(or all intervals of values of~$z$), certain events involving  $\Prb(e(S)=z)$ hold with
very high probability. We remark again that when we say an event holds w.v.h.p.\ for any~$z$,
this implies that w.v.h.p.\ the event holds for all $z$ at the same time by a simple
union bound argument. We will thus pay no particular attention to the order of ``w.v.h.p.'' and
``for all~$z$''.

The proof of \cref{thm:llt} starts from the central limit theorem (\cref{thm:clt}), and from there narrows down
to point probabilities via a smoothing argument. For a convenient starting point, we reformulate the central limit
theorem to deal with probabilities of $e(S)$ being in an interval $[z_0,z_1]$.

\begin{corollary}\label{cor:clt}
 For any $\eps>0$, w.v.h.p.\ in $G=G_{n,M}$ we have for any $z_0$, $z_1$ that
 \[
  \big|\Prb_S(e(S)\in [z_0,z_1])-\Prb_S(Z \in [z_0,z_1])\big|\le n^{-1/4+\eps},
 \]
 where $Z\sim \Nor(KM/N,\lambda n^2)$.
\end{corollary}
\begin{proof}
This follows from \cref{thm:clt} together with the fact that
$\Prb_S(e(S)\in [z_0,z_1])=\Prb(e(S)\le z_1)-\Prb(e(S)<z_0)=\Prb(e(S)\le z_1)-\Prb(e(S)\le z'_0)$,
where $z'_0<z_0$ is sufficiently close to~$z_0$.
\end{proof}

We also record a (weaker) asymptotic upper bound on these probabilities of intervals,
which will come in handy in the smoothing procedure that is central in the argument to come. 

\begin{corollary}\label{cor:initialbdd}
 For $p=p(n)\in[\delta,1-\delta]$ and\/ $\eps>0$, there is a constant\/ $C>0$
 such that w.v.h.p.\ in $G_{n,p}$, for any $z$ and any $r$ with $r\ge n^{3/4+\eps}$, 
 \[
  \Prb_S\big(e(S) \in [z,z+r]\big)\le Cr/n.
 \]
\end{corollary}
\begin{proof}
With $Z$ as before, \cref{cor:clt} gives that for a fixed $M=M(n)$ w.v.h.p.\ in $G_{n,M}$,
we have $\Prb(e(S) \in [z,z+r])\le \Prb(Z\in [z,z+r]) + n^{-1/4+\eps}$. Thus, as the pdf of $Z$
is bounded by $1/(n\sqrt{2\pi\lambda})$ (see~\cref{lem:normal}),
\begin{align*}
 \Prb\big(e(S) \in [z,z+r]\big)\le \tfrac{r}{n\sqrt{2\pi\lambda}}+n^{-1/4+\eps}\le Cr/n
\end{align*} 
for a suitable constant $C$ and provided $r\ge n^{3/4+\eps}$. Since this holds in $G_{n,M}$ for any
$M$ with the same constant~$C$, revealing the number of edges in $G_{n,p}$, yields the same
assertion in $G_{n,p}$, noting that the probability that $M/N\notin[\delta/2,1-\delta/2]$ is $o(1/n)$.
\end{proof}

In addition to the preceding corollaries, the other main ingredients needed for our smoothing
method are certain concentration results which we prove in \cref{subsec:lltconc}.
The smoothing procedure itself is then described and established in \cref{subsec:lltsmoothing},
and we apply it to prove \cref{thm:llt} in \cref{subsec:lltpf}. While the target local
limit theorem is for $G_{n,M}$, it is convenient to work for the most part in $G_{n,p}$ for
some $p=p(n)\in[\delta,1-\delta]$. We only return to $G_{n,M}$
for the final proof, at which point we transfer results between models as needed.

\subsection{Concentration of \texorpdfstring{$e(S)-e(S')$}{e(S)-e(S')} in \texorpdfstring{$G_{n,p}$}{Gnp}}
\label{subsec:lltconc}

With the justification at the end of the preceding section, we now work in $G_{n,p}$
for some $p=p(n)\in[\delta,1-\delta]$ and assume $V(G_{n,p})=[n]$. 

To analyse our random model, we view our random set $S$ as being constructed by taking a union $S=S'\cup T$
of a random $(k-t)$-set $S'\subseteq [n]$ and a random $t$-set $T\subseteq [n]\setminus S'$, where $t=t(n)$
will be of order $1\ll t\ll n$. We will condition on $S'$ and compare probabilities involving $e(S)-e(S')$
to a better-behaved binomial random variable~$Y$. Henceforth let $Y\sim \Bin((k-t)t+\binom{t}{2},p)$
and note that for fixed $S=S'\cup T$, in $G_{n,p}$, $e(S)-e(S')$ is distributed like~$Y$. 

The next two lemmas show that in $G_{n,p}$, for any interval $A$ of appropriate length and for almost all $S'$,
$\Prb_T(e(S)-e(S') \in A)$ is concentrated around its mean $\Prb(Y \in A)$. We first do this in
\cref{lem:BernsteinAppl} with $S'$ fixed and $T$ chosen from a family of disjoint sets in
$[n]\setminus S'$. In this setting, the relevant edge sets are independent allowing us to use simple
concentration techniques. The more general case is then contained in \cref{lem:markov}.

\begin{lemma}[$T$ from disjoint family]\label{lem:BernsteinAppl}
 Assume $p,k/n\in[\delta,1-\delta]$.
 Let\/ $\beta,\eps>0$ be constants and let\/ $a,t\in\N$ be such that\/ $t\le k/2$ and
 \begin{equation}\label{eq:bernstein_condition}
  a \le t^{3/2}n^{2\beta-1/2}.
 \end{equation}
 Let\/ $S'\subseteq [n]$ be fixed with $|S'|=k-t$. Define $m=\lfloor\frac{n-k+t}{t}\rfloor$ and let\/
 $\cT=\{T_i\}_{i\in [m]}$ be a family of pairwise disjoint\/ $t$-sets in $[n]\setminus S'$.
 Let\/ $T$ be a uniformly chosen random element of\/ $\cT$ and\/ $A=[z,z+a)$ an interval.
 If\/ $n$ is sufficiently large, then with probability at least $1-e^{-n^\eps}$ in $G_{n,p}$ we have 
 \[
  \big|\Prb_T\big(e(S'\cup T)-e(S')\in A\big)-\Prb(Y \in A)\big|\le \Delta_A,
 \]
 where
 \begin{align}\label{eq:Delta}
  \Delta_A=
  \begin{cases}
   e^{-n^\eps}/m,&\text{if\/ }\Prb(Y\in A)\le e^{-n^\eps}/m;\\
   n^{\beta+\eps}/m,&\text{otherwise.}
  \end{cases}
 \end{align}
\end{lemma}
\begin{proof}
In this proof, it will be easier to work with counts of sets rather than probabilities.
To that end, let $N_A$ be the number of sets $T\in \cT$ such that $e(S'\cup T)-e(S')\in A$.
Then we have $\mu_A = \E_G[N_A]=m\Prb(Y \in A)$, and hence
\[
 \Prb_G\big(|\Prb_T\big(e(S'\cup T)-e(S')\in A\big)-\Prb(Y \in A)|>\Delta_A\big)
 =\Prb_G\big(|N_A-\mu_A|>m\Delta_A\big).
\]
We proceed by working with the expression on the right hand side.
If $\Prb(Y \in A)\le e^{- n^\eps}/m$, then $\mu_A \le e^{- n^\eps}$,
so Markov's inequality gives 
\[
 \Prb_G(|N_A-\mu_A|>e^{- n^\eps})\le
 \Prb_G(N_A\ge 1)\le \E_G(N_A)\le e^{- n^\eps}.
\]
We now turn to the general case. Clearly we may assume $\delta\le\frac12$, so that
by our assumptions on $p$ and $k$ and using that $t\le k/2$ and $k\ge \delta n$ we have
\begin{equation}\label{eq:varY}
 \Var(Y)=p(1-p)\big((k-t)t+{\tbinom{t}{2}}\big)\ge \delta(1-\delta)\cdot\tfrac{kt}{2}\ge\tfrac{1}{4}\delta^2 n t.
\end{equation}
Then, by \cref{lem:bino},
\[
 p':=\Prb(Y\in A)\le \sqrt{\frac{\pi}{8\Var(Y)}}\cdot a\le \frac{2a}{\delta\sqrt{nt}}
\] 
and we have $N_A\sim \Bin(m,p')$. Hence, by our assumption~\eqref{eq:bernstein_condition},
\[
 \mu_A= mp'\le \frac{n}{t}\cdot\frac{2a}{\delta\sqrt{nt}}\le 2\delta^{-1}n^{2\beta}.
\] 
It then follows from Bernstein's inequality (\cref{Bernstein}) that, for large enough~$n$,
\[
 \Prb_G\big(|N_A-\mu_A|>n^{\beta+ \eps}\big)\le
 2\exp\left(\frac{-n^{2 \beta+2\eps}}{2(\mu_A+n^{\beta+\eps}/3)}\right)
 \le e^{-n^{\eps}}.\qedhere
\]
\end{proof}

We now show that the assertion of \cref{lem:BernsteinAppl} still holds for almost all $S'$ when
we take a fully random $T\subseteq[n]\setminus S'$. This is a straightforward application of Markov's inequality.

\begin{lemma}[Concentration]\label{lem:markov}
 Assume $p$, $k$, $\beta$, $\eps$, $a$ and\/ $t$ are as in \cref{lem:BernsteinAppl}. For a fixed\/ $(k-t)$-set $S'$
 let\/ $T$ be a random $t$-set in $[n] \setminus S'$. Then w.v.h.p.\ in $G_{n,p}$, for any
 interval\/ $A$ as above, at most an $e^{-n^{\eps}/4}$ proportion of the $(k-t)$-sets $S'$ fail to satisfy 
 \begin{equation}\label{MarkovTinequ}
  \big|\Prb_T(e(S',T)+e(T)\in A)-\Prb(Y\in A)\big|\le \Delta_A + e^{-n^{\eps}/2}.
 \end{equation}
\end{lemma}
\begin{proof}
We begin with a fixed set~$S'$. Picking a random $T$ is equivalent to choosing uniformly a random (almost complete)
partition of $[n]\setminus S'$ into $t$-sets $\cT=\{T_i\}_{i\in[m]}$ and a uniformly random $T\in\cT$. Say a family
$\cT$ is \emph{good} if, when choosing $T$ u.a.r.\ from~$\cT$, we have
$|\Prb_T\big(e(S',T)+e(T)\in A\big)-\Prb(Y\in A)|\le \Delta_A$, otherwise we say $\cT$ is \emph{bad}.

When choosing a random~$\cT$, we have $\E_G[\Prb_\cT(\cT\text{ is bad})]<e^{-n^{\eps}}$ by
\cref{lem:BernsteinAppl}. Applying Markov's inequality then gives
\[
 \Prb_G\big(\Prb_\cT(\cT\text{ is bad})\ge e^{-n^{\eps}/2}\big)
 \le \E_G[\Prb_\cT(\cT\text{ is bad})]/e^{-n^{\eps}/2}\le e^{-n^{\eps}/2}.
\]
Thus, when choosing $\cT$ and $T\in\cT$ as described above, with probability at
least $1-e^{-n^{\eps}/2}$ in $G$ we have 
\begin{align*}
 \big|\Prb_T(e(S',T)+e(T)\in A)-\Prb(Y\in A)\big|
 &\le \Prb_\cT(\cT\text{ is good})\Delta_{A} + \Prb_\cT(\cT\text{ is bad})\\
 &\le \Delta_A + e^{-n^{\eps}/2}.
\end{align*}
Writing $N_e$ for the number of $S'$ that violate \eqref{MarkovTinequ},
this gives $\E_G[N_e]\le e^{-n^{\eps}/2}\binom{n}{k-t}$.
By a second application of Markov's inequality, we now have
\[
 \Prb_G\big(N_e>e^{-n^{\eps}/4}\tbinom{n}{k-t}\big)
 \le e^{-n^{\eps}/4}
\]
as desired.
\end{proof}

\subsection{Smoothing}
\label{subsec:lltsmoothing}

For a (fixed) random graph~$G$, the only properties that we will now need to analyse the distribution of $e(S)$ are the
bounds from the CLT reformulation in \cref{cor:initialbdd} and the concentration property given in \cref{lem:markov}.
We see directly from these two statements that w.v.h.p.\ $G_{n,p}$ (with $p\in[\delta,1-\delta]$) satisfies both of these.
The goal of this subsection is to show that for such $G$ the probability that $e(S)=z$, or the
\emph{weight} of~$z$, is evenly distributed for points $z$ within intervals of a suitable length in the following sense. 

\begin{lemma}[Smoothing Lemma]\label{lem:smoothing}
 Fix $\beta\in(0,\frac{1}{10})$ and\/ $\eps\in(0,\frac{1-10\beta}{4})$.
 Then there exists an integer $a\in[n^{1-5\beta/2-\eps},n^{1-5\beta/2}]$ such that
 w.v.h.p.\ in $G_{n,p}$, for all\/ $z$, we have
 \[
  \big|\Prb_S\big(e(S)=z\big)-a^{-1}\Prb_S\big(e(S)\in [z,z+a)\big)\big| \le n^{-1-\beta+2\eps}.
 \]
\end{lemma}

To see how this fits into our main outline, note that the central limit theorem gives us the measure
of any interval with error $O(n^{-1/4+\eps})$. If we then work on an interval of length about $n^{1-5\beta/2}$,
the smoothing argument spreads out this error to $O(n^{-5/4+5\beta/2+\eps})$ at
each point and introduces an additional $O(n^{-1-\beta+2\eps})$ error. There is also an additional
small error from the variation of the Gaussian, and altogether this leads to the claimed local limit theorem
when $\beta$ is chosen appropriately.

\cref{lem:smoothing} will be established using an iterative procedure to gradually prove
smoothing statements for intervals of decreasing length (descending to intervals of length~1).
In each step, we consider any interval $A$ of length $r$ from the previous iteration,
and show that for any two subintervals $A_1$ and $A_2$ of $A$ with the same length~$a\ll r$, their
weights cannot differ too much from each other. This is the content of \cref{lem:AA'}.
Combining all these smoothing statements, we obtain that the probabilities at two single points in any
interval of length given by the first iteration step do not differ too much, giving \cref{lem:smoothing}.
The condition that \cref{lem:AA'} imposes on intervals of length $r$ to get an iteration
started is rather weak -- we ask that weights of those intervals are at most $O(r/n)$. As we will
see in \cref{cor:generaliter}, this condition follows directly from the assertion of the previous application
of \cref{lem:AA'}; at the first iteration, it follows from the CLT in the form of \cref{cor:initialbdd}.

We shall use an iterative sequence for the above procedure. Each step requires a suitable choice of parameters
$a$ and $r$ as above, as well as $t$ from \cref{subsec:lltconc}. The conditions of \cref{lem:AA'} impose
the following constraints. For a given constant $\beta\in(0,\frac{1}{10})$,
let us say that the triple $(a,r,t)$ of integers is
\textit{valid} (\textit{for} $\beta$) provided that
\begin{enumerate}[label=(\roman*)]
 \item $a\le r$,\label{valid:a_r}
 \item $a=ct^{3/2}n^{2\beta-1/2}$ for some $\tfrac12\le c\le 1$, and\label{valid:nat} 
 \item $r\le n^{-\beta}\sqrt{nt}$.\label{valid:r}
\end{enumerate}

\begin{remark}\label{rem:constraint}
In Lemmas~\ref{lem:BernsteinAppl} and~\ref{lem:markov} we already required the $c\le1$ condition in~\ref{valid:nat}
explicitly while the requirement that $t \le k/2$ follows for large $n$ from the $c\ge\frac12$ part of~\ref{valid:nat}
(as (i)--(iii) imply $ct\le n^{1-3\beta}$ while $k\ge \delta n$). We also note for later
that if $r^3n^{-2+5\beta}\le a\le r$ then one can set $t=\lceil (a^2n^{1-4\beta})^{1/3}\rceil$ to satisfy (i)--(iii)
when $n$ is large.
\end{remark}

\begin{lemma}[Difference Lemma]\label{lem:AA'}
 Let\/ $\beta\in(0,\frac{1}{10})$, $\eps>0$ and let\/ $(a,r,t)$ be valid for~$\beta$. Let\/ $S'$ and\/ $S$ be uniformly
 and independently chosen random subsets of\/ $[n]$ of sizes $k-t$ and\/ $k$ respectively.
 Suppose that there is a constant\/ $C$ such that w.v.h.p.\ in $G_{n,p}$ we have 
 \begin{equation}\label{boundedprob}
  \Prb_{S'}\big(e(S')\in [z,z+r)\big)\le\tfrac{Cr}{n}
 \end{equation}
 for all\/~$z$. Then w.v.h.p.\ in $G_{n,p}$ we have 
 \begin{equation}\label{eq:diff_conclusion}
  \big|\Prb_S\big(e(S)\in [z_1,z_1+a)\big)-\Prb_S\big(e(S)\in [z_2,z_2+a)\big)\big|\le an^{-1-\beta+\eps}
 \end{equation}
 whenever $|z_1-z_2|\le r$.
\end{lemma}

\begin{proof}
We can assume that $n$ is large, so that $(a,r,t)$ being valid for $\beta$ implies $t\le k/2$.
To rewrite the quantity we wish to bound, we couple our choices of $S$ and~$S'$:
generate a random set $S$ in parts by taking a random set $S'$ of size $k':= k-t$ as well as a random set
$T \subseteq [n]\setminus S'$ of size~$t$, and then defining $S:=S'\cup T$ (as in \cref{subsec:lltconc}).
Let  $A_1:=[z_1,z_1+a)$ and $A_2:=[z_2,z_2+a)$ and note that $|z_1-z_2|\le r$ by assumption.
We call $S'$ \emph{good} if \eqref{MarkovTinequ} holds for both $A_1$ and~$A_2$, otherwise $S'$ will be called \emph{bad}.
If we condition on $S'$ and $S'$ is good, then
\begin{align*}
 \big|\Prb(e(S)\in A_1\mid S')-&\Prb(e(S)\in A_2\mid S')\big|\\
 &\le\big|\Prb_T\big(e(S)-e(S') \in A_1-e(S')\big) - \Prb_Y\big(Y \in A_1-e(S')\big)\big|\\
 &\quad+\big|\Prb_T\big(e(S)-e(S') \in A_2-e(S')\big) - \Prb_Y\big(Y \in A_2-e(S')\big)\big|\\
 &\quad+\big|\Prb_Y\big(Y \in A_1-e(S')\big) - \Prb_Y\big(Y \in A_2-e(S')\big)|\\
 &\le \Delta_{A_1-e(S')}+\Delta_{A_2-e(S')} + 2e^{-n^{\eps}/2}+P_Y(e(S')),
\end{align*}
where $Y\sim\Bin(t(k-t)+\binom{t}{2},p)$ and
\[
 P_Y(e) := \big|\Prb\big(Y \in A_1-e\big) - \Prb\big(Y \in A_2-e\big)\big|.
\]
The conclusion of \cref{lem:markov} means that w.v.h.p.\ in $G_{n,p}$, $\Prb_{S'}(S'\text{ is bad})\le 2e^{-n^{\eps}/4}$, so
\begin{align} 
 \big|\Prb(e(S)\in A_1)&-\Prb(e(S)\in A_2)\big|\notag\\
 &\le \Prb(S'\text{ is bad})+\E_{S'}\big[\Delta_{A_1-e(S')}+\Delta_{A_2-e(S')} + 2e^{-n^{\eps}/2}+P_Y(e(S'))\big]\notag\\
 &\le 4e^{-n^{\eps}/4}+\E_{S'}\big[\Delta_{A_1-e(S')}+\Delta_{A_2-e(S')}+P_Y(e(S'))\big]\notag\\
 &\le 4e^{-n^\eps/4} + \sum_{0\le e\le \binom{k'}{2}}\Prb(e(S')=e)\big(\Delta_{A_1-e}+\Delta_{A_2-e} + P_Y(e) \big).
 \label{eq:sum_e}
\end{align}
Let $I$ be the set of $e$ such that there is a $z\in A_1\cup A_2$ with 
\begin{equation}\label{eq:contributing_e}
 |z-e-\E[Y]|\le\sqrt{kt}\,n^{\eps}.
\end{equation}
When $e \notin I$, we have exponential bounds for all terms of~\eqref{eq:sum_e}: using \cref{Bernstein}
we have for $i=1,2$ and sufficiently large~$n$
\begin{align*}
 \Prb\big(Y\in A_i-e\big)
 &\le \Prb\big(|Y-\E[Y]|>\sqrt{kt}\,n^{\eps}\big)\\
 &\le 2\exp\bigg(-\frac{ktn^{2\eps}}{2\big(\big((k-t)t+\binom{t}{2}\big)p(1-p)+\sqrt{kt}\,n^{\eps}/3\big)}\bigg)\\
 &\le 2\exp\bigg(-\frac{ktn^{2\eps}}{kt+\sqrt{kt}\,n^{\eps}}\bigg)\\
 &\le \exp(-n^\eps)/m,
\end{align*}
where $m=\lfloor\frac{n-k+t}{t}\rfloor$ as in \cref{lem:BernsteinAppl}. Thus we are in the first case of~\eqref{eq:Delta}, and so
$\Delta_{A_1-e},\Delta_{A_2-e}=\exp({-n^{\eps}})/m$, and also $P_Y(e)\le\exp(-n^{\eps})/m$. Hence
\[
 4e^{-n^\eps/4}+\sum_{e\notin I} \Prb(e(S')=e)\big(\Delta_{A_1-e}+\Delta_{A_2-e} + P_Y(e) \big)
 \le 4e^{-n^\eps/4}+3e^{-n^\eps}/m\le 5e^{-n^{\eps}/4}.
\]
We now turn to the contribution from the $e\in I$. For these $e$, in \eqref{eq:Delta} we have
$\Delta_{A_i-e}\le n^{\beta+\eps}/m$. In addition, \cref{lem:bino} gives a bound on $P_Y(e)$, so
that by \eqref{eq:varY} and the fact that $m\ge (n-k)/t\ge \delta n/t$ we have
\begin{align*}
 \Delta_{A_1-e}+\Delta_{A_2-e} + P_Y(e)
 &\le 2n^{\beta+\eps}/m+\frac{\pi a|z_2-z_1|}{4\Var(Y)}\le 2\delta^{-1}n^{\beta+\eps-1}t+\frac{\pi ar}{\delta^2 nt}\\
 &\le\tfrac{a}{\sqrt{nt}}\big(4\delta^{-1}n^{-\beta+\eps}+\pi\delta^{-2}n^{-\beta}\big)
 \le \tfrac{a}{\sqrt{nt}}\cdot 5\delta^{-1} n^{-\beta+\eps},
\end{align*}
where we have used conditions \ref{valid:nat} and~\ref{valid:r} and assumed $n$ sufficiently large.

To obtain a bound on the total weight of all $e \in I$, we first observe that, for such~$e$, the condition
in \eqref{eq:contributing_e} is in fact satisfied by one of $z_1$, $z_1+a$, $z_2$ or $z_2+a$. Indeed, when
$e+\E[Y]\notin A_1\cup A_2$, one of $z_1$, $z_1+a$, $z_2$, $z_2+a$ is at least as close to $e+\E[Y]$ as any
element of $A_1 \cup A_2$. Else, if $e+\E[Y]$ is an element of, say, $A_1$, the validity of $(a,r,t)$
(conditions~\ref{valid:a_r} and~\ref{valid:r}) implies that $a \le r\le \sqrt{nt} \le \sqrt{kt}\,n^{\eps}$,
so it follows that any $z \in A_1$, and in particular $z_1$, satisfies~\eqref{eq:contributing_e}. We thus get  
\begin{align*}
 \Prb_{S'}(e(S')\in I)&\le \sum_{z\in\{z_1,z_1+a,z_2,z_2+a\}}\Prb_{S'}\big(|z-e(S')-\E[Y]|\le \sqrt{kt}\,n^{\eps}\big)\\
 &\le 8C(2\sqrt{kt}\,n^\eps+1)/n=O(\sqrt{nt}\cdot n^{\eps-1}).
\end{align*}
For the last inequality, observe that~\eqref{boundedprob} implies that for all $r'\ge r$,
$\Prb(e(S') \subseteq[z,z+r'))\le 2Cr'/n$ and that this can be applied as $r\le \sqrt{kt}\,n^{\eps}$.

Putting all of this together and using that $(a,r,t)$ is valid (conditions~\ref{valid:nat} and~\ref{valid:r})
in the penultimate step, we obtain
\begin{align*}
 |\Prb(e(S)\in A_1)-\Prb(e(S) \in A_2)|
 &\le 5e^{-n^{\eps}/4}+\sum_{e\in I}\Prb(e(S')=e)\big(\Delta_{A_1-e}+\Delta_{A_2-e} + P_Y(e) \big)\\
 &= 5e^{-n^{\eps}/4}+O(\sqrt{nt}\cdot n^{\eps-1})\cdot \tfrac{a}{\sqrt{nt}}\cdot 5\delta^{-1} n^{-\beta+\eps}\\
 &=O(an^{-\beta-1+2\eps}).
\end{align*}
As this holds for any $\eps>0$, the result follows for sufficiently large $n$ on replacing $\eps$ by $\eps/3$, say.
\end{proof}

We now observe that the conclusion of \cref{lem:AA'} yields~\eqref{boundedprob} for the smaller interval of size~$a$,
which will allow us to kick off the next iteration step. 

\begin{corollary}\label{cor:generaliter}
 Under the notation and assumptions of \cref{lem:AA'}, there is a constant $C'>0$ such that w.v.h.p.\ in $G_{n,p}$
 we also have $\Prb\big(e(S)\in [z,z+a)\big) \le \frac{C'a}{n}$ for all\/~$z$.
\end{corollary}
\begin{proof}
We look at $[z,z+a)$ for any fixed~$z$. Consider the interval $[z,z+r)$, and let $\cI$ be a maximal collection of disjoint
intervals of length $a$ in $[z,z+r)$, so that $|\cI|= \lfloor r/a\rfloor\ge r/2a$. W.v.h.p., for each $A\in\cI$, we have by
\cref{lem:AA'} that $|\Prb(e(S)\in [z,z+a))-\Prb(e(S)\in A)| \le a/n$ (as without loss of generality $\eps<\beta$).
Thus this assertion and
$\Prb(e(S')\in [z,z+r))\le \frac{Cr}{n}$ hold at the same time w.v.h.p.\ and they imply
$\Prb(e(S)\in [z,z+a)) \le \frac{C'a}{n}$ for $C'=2C+1$. Indeed, supposing otherwise gives 
\[
 \Prb\big(e(S) \in [z,z+r)\big) \ge  \sum_{A\in \cI} \Prb(e(S)\in A)
  > \left(\tfrac{C'a}{n} -\tfrac{a}{n}\right) \cdot \tfrac{r}{2a}=\tfrac{Cr}{n},
\]
which contradicts our assumption.
\end{proof}

We are now ready to prove the main smoothing lemma.

\begin{proof}[Proof of \cref{lem:smoothing}]
Set $a_0=1$ and inductively define for $j\ge0$,
$t_j=\lceil (a_j^2 n^{1-4\beta})^{1/3}\rceil$ and $a_{j+1}=\lfloor (a_jn^{2-5\beta})^{1/3}\rfloor$.
It can be easily verified that for all~$j$, $(a_j,a_{j+1},t_j)$ is valid for~$\beta$.
Indeed, $a_j$ will be (just less than) $n^{(1-5\beta/2)(1-3^{-j})}$, so
$a_j\le a_{j+1}$. Also $n^{-\beta}\sqrt{nt_j}\ge (a_jn^{2-5\beta})^{1/3} \ge  a_{j+1}$,
and $a_j\le t_j^{3/2}n^{2\beta-1/2}=(1+o(1))a_j$. 
Also, there exists some~$j_0$, depending only on~$\eps$,
such that for sufficiently large~$n$, $a_{j_0}\in [n^{1-5\beta/2-\eps},n^{1-5\beta/2}]$.
This will be our~$a$.

Now let $T_j=\sum_{i=0}^{j-1} t_i$ and note that as $j_0$ is a constant and $t_j\le \lceil n^{1-3\beta}\rceil$,
we have that for sufficiently large $n$ that $T_{j_0+1}\le k/2$. We now prove by reverse induction on $j$ that
there is a constant $C_j$ such that for any $k'\in [k-T_j,k]$, if we choose a set $S$ of size $k'$ uniformly at
random, then w.v.h.p.\ in $G_{n,p}$ we have for all~$z$,
\begin{equation}\label{eq:alt1}
 \Prb_S\big(e(S)\in [z,z+a_j)\big)\le C_ja_j/n
\end{equation}
and (for $j<j_0$) 
\begin{equation}\label{eq:alt2}
 \big|a_j^{-1}\Prb_S\big(e(S)\in[z,z+a_j)\big)-a_{j+1}^{-1}\Prb_S\big(e(S)\in[z,z+a_{j+1})\big)\big|\le n^{-1-\beta+\eps}.
\end{equation}
Indeed \eqref{eq:alt1} holds for $j=j_0$ by \cref{thm:clt} provided $a_{j_0}\ge n^{3/4+\eps'}$ for some $\eps'>0$,
and this holds as $1-5\beta/2-\eps>3/4$. Then for $j=j_0-1,\dots,0$ in turn, if $S$ is a uniformly chosen set
of size $k'\in[k-T_j,k]$ and $S'$ is a uniformly chosen set of size $k'-t_j\in[k-T_{j+1},k]$, then
\eqref{eq:alt1} applied with $S$ replaced with $S'$ and $j$ replaced with $j+1$ implies \eqref{eq:alt2} by \cref{lem:AA'}.
\cref{cor:generaliter} then implies \eqref{eq:alt1} for $S$ and~$j$.

Hence we have for all $z$, and $S$ uniformly random of size~$k$,
\[
 \big|\Prb\big(e(S)=z\big)-a_{j_0}^{-1}\Prb\big(e(S)\in[z,z+a_{j_0})\big)\big|
 \le j_0n^{-1-\beta+\eps}\le n^{-1-\beta+2\eps}.\qedhere
\]
\end{proof}

\subsection{Proof of Theorem~\ref{thm:llt}}\label{subsec:lltpf}

Combining the control we have from \cref{lem:smoothing}, which relates single-point weights to weights of long intervals,
with \cref{cor:clt} which gives us the weight of long intervals, we can now obtain the single-point weights up to a small
error term as desired.

\begin{proof}[Proof of \cref{thm:llt}]
\cref{lem:smoothing} states that w.v.h.p.\ $G_{n,p}$ satisfies
\[
 \big|\Prb\big(e(S)=z\big) - a^{-1}\Prb\big(e(S)\in [z,z+a)\big)\big| \le n^{-1-\beta+\eps},
\]
for all~$z$, where $a$ is some integer satisfying 
\[
 n^{1-5\beta/2-\eps}\le a\le n^{1-5\beta/2}.
\]
Setting $p=M/N$, there is a polynomial probability (on the order of $1/n$) that
$e(G_{n,p})=M$. In that case we have $(G_{n,p}\mid e(G_{n,p})=M)\sim G_{n,M}$, so the conclusion of \cref{lem:smoothing}
transfers to $G_{n,M}$ again with very small failure probability. Using this in the first step and \cref{cor:clt}
in the second, and with $Z$ and $\phi$ as in the theorem statement, we obtain
\begin{align*} 
 \Prb\big(e(S)=z\big)
 &=a^{-1}\Prb\big(e(S) \in [z,z+a)\big)+O(n^{-1-\beta+\eps})\\
 &=a^{-1}\Prb\big(Z\in [z,z+a)\big) + O(a^{-1}n^{-1/4+\eps}+n^{-1-\beta+\eps})\\
 &=a^{-1}\int_z^{z+a} \varphi(t)\,dt + O(a^{-1}n^{-1/4+\eps}+n^{-1-\beta+\eps})\\
 &=\varphi(z)+O(a/n^2+a^{-1}n^{-1/4+\eps}+n^{-1-\beta+\eps})\\
 &=\varphi(z)+O(n^{-1-5\beta/2}+n^{-5/4+5\beta/2+2\eps}+n^{-1-\beta+\eps}),
\end{align*}
where the penultimate step follows from \cref{lem:normal} and the fact that $\Var(Z)=\Theta(n^2)$.
To optimise the exponent, we take $\beta=1/14$, giving
\[
  \Prb(e(S)=z)=\varphi(z)+O(n^{-1-1/14+2\eps}).\qedhere
\]
\end{proof}

\bibliographystyle{abbrv}
\bibliography{limittheorems}

\begin{thebibliography}{10}

\bibitem{Bernstein}
S.~N. Bernstein.
\newblock {On a modification of Chebyshev’s inequality and of the error
  formula of Laplace}.
\newblock {\em Ann. Sci. Inst. Sav. Ukraine, Sect. Math.}, \textbf{4}:38--48,
  1924.

\bibitem{bhattacharya}
B.~B. Bhattacharya, P.~Diaconis, and S.~Mukherjee.
\newblock Universal limit theorems in graph coloring problems with connections
  to extremal combinatorics.
\newblock {\em The Annals of Applied Probability}, pages 337--394, 2017.

\bibitem{BG93}
E.~Bolthausen and F.~G\"otze.
\newblock {The rate of convergence for multivariate sampling statistics}.
\newblock {\em The Annals of Statistics}, \textbf{21}:1692 -- 1710, 1993.

\bibitem{calkin1992subgraph}
N.~Calkin, A.~Frieze, and B.~D. McKay.
\newblock On subgraph sizes in random graphs.
\newblock {\em Combinatorics, Probability and Computing}, 1(2):123--134, 1992.

\bibitem{chatterjee}
S.~Chatterjee.
\newblock A new method of normal approximation.
\newblock {\em The Annals of Probability}, pages 1584--1610, 2008.

\bibitem{dejong}
P.~De~Jong.
\newblock A central limit theorem for generalized quadratic forms.
\newblock {\em Probability Theory and Related Fields}, 75(2):261--277, 1987.

\bibitem{efav}
P.~Erd{\H{o}}s.
\newblock Some of my favourite problems in various branches of combinatorics.
\newblock {\em Le Matematiche}, 47(2):231--240, 1992.

\bibitem{esz}
P.~Erd{\H{o}}s and A.~Szemer{\'e}di.
\newblock On a {Ramsey} type theorem.
\newblock {\em Periodica Mathematica Hungarica}, 2(1-4):295--299, 1972.

\bibitem{hall}
P.~Hall.
\newblock Central limit theorem for integrated square error of multivariate
  nonparametric density estimators.
\newblock {\em Journal of multivariate analysis}, 14(1):1--16, 1984.

\bibitem{kwan}
M.~Kwan, A.~Sah, L.~Sauermann, and M.~Sawhney.
\newblock Anticoncentration in {Ramsey graphs} and a proof of the
  {Erd\H{o}s}--{McKay} conjecture.
\newblock In {\em Forum of Mathematics, Pi}, volume~11, page e21, 2023.

\bibitem{RR}
G.~Reinert and A.~Rollin.
\newblock {Multivariate normal approximation with Stein's method of
  exchangeable pairs under a general linearity condition}.
\newblock {\em The Annals of Probability}, \textbf{37}:2150--2173, 2009.

\bibitem{rotar1974some}
V.~Rotar’.
\newblock Some limit theorems for polynomials of second degree.
\newblock {\em Theory of Probability \& Its Applications}, 18(3):499--507,
  1974.

\bibitem{tao2015random}
T.~Tao and V.~Vu.
\newblock {Random matrices: universality of local spectral statistics of
  non-Hermitian matrices}.
\newblock {\em The Annals of Probability}, \textbf{43}(2):782--874, 2015.

\end{thebibliography}

\appendix

\section{Technical tools}\label{sec:technical}

In this section, we record several technical results that will be used in proofs of our main theorems.
While none of them are new, we provide proofs for completeness in cases where a convenient reference is lacking. 

We start with two concentration inequalities, namely Bernstein's inequality and a version of the Azuma--Hoeffding
inequality with a relaxed difference bound condition.

\begin{theorem}[Bernstein~\cite{Bernstein}]\label{Bernstein}
 Let\/ $X_1,\dots,X_n$ be independent random variables taking values in $[0,1]$ and let\/ $X:=\sum_{i=1}^nX_i$
 and\/ $\sigma^2:=\Var(X)=\sum_{i=1}^n\Var(X_i)$. Then for any $t\ge 0$ we have
 \[
  \Prb\big(|X-\E[X]|>t\big)\le2\exp\bigg(-\frac{t^2}{2(\sigma^2+t/3)}\bigg).
 \]
\end{theorem}

\begin{lemma}[Azuma--Hoeffding, see e.g.\ \cite{tao2015random}, Proposition 34]\label{AH}
 Let\/ $X=X(Z_1,\dots,Z_n)$ be a random variable depending on independent variables
 $Z_1,\dots,Z_n$ and suppose $c_1,\dots,c_n>0$ are such that
 \[
  \Prb\big(\exists i\colon\big|\E[X\mid Z_1,\dots Z_i]-\E[X\mid Z_1,\dots Z_{i-1}]\big|>c_i\big)\le\eps.
 \]
 Then
 \[
  \Prb\big(|X-\E[X]|\ge t\big)\le\eps+2\exp\big(-t^2/2\sigma^2\big),
 \]
 where $\sigma^2:=\sum_{i=1}^n c_i^2$.
\end{lemma}

The following is more general that we need as we will only be applying it in the case when $X$ is a Binomial random variable.
Nevertheless, we state it here in greater generality as the proof is no harder.

\begin{lemma}\label{lem:bino}
 Suppose $X=\sum_{i=1}^nX_i$ where $X_1,\dots,X_n$ are independent Bernoulli random variables. Then, for any $m,m'\in\Z$,
 \[
  \Prb(X=m)\le \sqrt{\frac{\pi}{8\Var(X)}}\quad
  \text{and}\quad\big|\Prb(X=m')-\Prb(X=m)\big|\le \frac{\pi}{4\Var(X)}|m'-m|.
 \]
\end{lemma}
\begin{proof}
For $p\in[0,1]$ and $\theta\in[-\pi,\pi]$ we have the inequality
\[
 \big|(1-p)+pe^{i\theta}\big|^2=(1-p)^2+p^2+2p(1-p)\cos\theta
 =1-4p(1-p)\sin^2\tfrac{\theta}{2}\le e^{-4p(1-p)\theta^2/\pi^2},
\]
where we used $|\sin\phi|\ge \frac{2}{\pi}|\phi|$ for $|\phi|\le \frac{\pi}{2}$
and $1-x\le e^{-x}$ for all $x\in\R$. Now, as the $X_i$ are independent,
\[
 \Prb(X=m)=\E[\one_{\{X=m\}}]=\E\frac{1}{2\pi}\int_{-\pi}^{\pi} e^{i\theta (X-m)}\,d\theta
 =\frac{1}{2\pi}\int_{-\pi}^{\pi}e^{-im\theta}\prod_j\E e^{i\theta X_j}\,d\theta.
\]
Hence, writing $p_j:=\Prb(X_j=1)$ and $\sigma^2:=\Var(X)=\sum p_j(1-p_j)$,
\begin{align*}
 \Prb(X=m)
 &=\frac{1}{2\pi}\int_{-\pi}^{\pi}e^{-im\theta}\prod_j\big((1-p_j)+p_je^{i\theta}\big)\,d\theta
  \le\frac{1}{2\pi}\int_{-\pi}^{\pi}\prod_je^{-2p_j(1-p_j)\theta^2/\pi^2}\,d\theta\\
 &\le\frac{1}{2\pi}\int_{-\infty}^\infty e^{-2\sigma^2\theta^2/\pi^2}\,d\theta
  =\frac{1}{2\pi}\cdot\sqrt{\frac{\pi}{2\sigma^2/\pi^2}}=\sqrt{\frac{\pi}{8\sigma^2}}.
\end{align*}
For the second inequality, we have
\[
 \Prb(X=m')-\Prb(X=m)=
 \frac{1}{2\pi}\int_{-\pi}^{\pi}(e^{-im'\theta}-e^{-im\theta})\prod_j\E e^{i\theta X_j}\,d\theta.
\]
Thus, as $|e^{-im'\theta}-e^{-im\theta}|\le|(m'-m)\theta|$,
\begin{align*}
 |\Prb(X=m')-\Prb(X=m)|
 &\le \frac{|m'-m|}{2\pi}\int_{-\pi}^{\pi}|\theta|\prod_je^{-2p_j(1-p_j)\theta^2/\pi^2}\,d\theta\\
 &\le \frac{|m'-m|}{\pi}\int_0^\infty \theta e^{-2\sigma^2\theta^2/\pi^2}\,d\theta\\
 &=\frac{|m'-m|}{\pi}\cdot \frac{\pi^2}{4\sigma^2}=\frac{\pi}{4\sigma^2}|m'-m|,
\end{align*}
as desired.
\end{proof}

\begin{remark} The constants in \cref{lem:bino} are not best possible, however it can be
seen by taking approximations to suitable Poisson distributions that we can't replace them with the
corresponding constants for the normal distribution given by the following.
\end{remark}

\begin{lemma}\label{lem:normal}
 If\/ $\varphi$ is a density function of a normal random variable with mean $\mu$ and variance~$\sigma^2$
 then, for all\/ $x,y\in\R$ we have
 \[
  \varphi(x)\le \frac{1}{\sigma \sqrt{2\pi}}\qquad\text{and}\qquad
  |\varphi(x)-\varphi(y)|\le\frac{|x-y|}{\sigma^2\sqrt{2\pi e}}.
 \]
\end{lemma}
\begin{proof}
The first statement is immediate as $\varphi(x)=e^{-(x-\mu)^2/2\sigma^2}/(\sigma\sqrt{2\pi})$. For the latter, again note
that the maximum of $x\mapsto |xe^{-x^2/2}|$ is $e^{-1/2}$ and occurs at $x=\pm1$ and hence
\[
 |\varphi'(x)|=\left|-\frac{1}{\sigma^2\sqrt{2\pi}}\frac{x-\mu}{\sigma}
 \exp\Big(-\Big(\frac{x-\mu}{\sigma}\Big)^2/2\Big)\right|
 \le\frac{1}{\sigma^2\sqrt{2\pi e}},
\]
with the maximum occurring at $(x-\mu)/\sigma=\pm1$. The statement then follows from the mean value theorem.
\end{proof}

\end{document}